\documentclass[11pt]{article}
\usepackage{graphicx}
\usepackage{pgfplots}

\usepackage[colorlinks=true,citecolor=blue]{hyperref}
\usepackage{natbib}
\usepackage{graphicx}
\usepackage{amsfonts}
\usepackage{amsmath}
\usepackage{amssymb}
\usepackage{url}
\usepackage{fancyhdr}
\usepackage{indentfirst}
\usepackage{enumerate}
\usepackage{titlesec}
\usepackage{amsthm}
\usepackage{dsfont}
\usepackage[misc]{ifsym}
\usepackage{subfigure}
\usepackage{float}
\usepackage{verbatim}
\theoremstyle{definition}

\newtheorem{definition}{Definition}

\theoremstyle{plain}
\newtheorem{theorem}{Theorem}
\newtheorem{lemma}{Lemma}
\newtheorem{proposition}{Proposition}
\newtheorem{corollary}{Corollary}
\theoremstyle{remark}
\newtheorem{remark}{Remark}

\usepackage{cases}

\theoremstyle{definition}

\def\N{\mathbb{N}}

\def\R{\mathbb{R}}

\usepackage[onehalfspacing]{setspace}

\usepackage{bm}
\usepackage{tikz-qtree}
\usepackage{tikz}

\begin{document}
	\title{Estimation of the Adjusted Standard-deviatile for Extreme Risks}
	\author{Haoyu Chen$^{[a]}$~ \thinspace \ Tiantian Mao$^{[a]}$~ \thinspace \
		Fan Yang$^{[b]}$ \\
		%EndAName
		$[a]$ {\small Department of Statistics and Finance, School of Management,
			School of Data Science}\\
		{\small \ University of Science and Technology of China, Hefei 230026, P. R.
			China}\\
		$[b]${\small \ Department of Statistics and Actuarial Science, University of
			Waterloo }\\
		{\small Waterloo, ON N2L 3G1, Canada }}
	\date{{\small October 10, 2023}}
	\maketitle
	
	\begin{abstract}
	{In this paper, we modify the Bayes risk for the expectile, the so-called variantile risk measure, to better capture extreme risks. The modified risk measure is called 
                the adjusted standard-deviatile.}  First, we derive the asymptotic expansions of the adjusted standard-deviatile. Next, based on the first-order asymptotic expansion, we propose two efficient estimation methods for the adjusted standard-deviatile at intermediate and extreme levels. {By using techniques from extreme value theory, the asymptotic normality is proved for both estimators for independent and identically distributed observations and for $\beta$-mixing time series, respectively.} Simulations and real data applications are conducted to examine the performance of the proposed estimators.  
		
		\bigskip
		
	\end{abstract}

	\section{Introduction}
	
Risk measures can be connected through the risk quadrangle proposed in
\cite{rockafellar2013fundamental}. For a random variable $X$, a risk
quadrangle is defined by the relationships between a risk $\mathcal{R}$, a
deviation $\mathcal{D}$, an error $\mathcal{E}$ and a statistic $\mathcal{S}$
as
\[
\mathcal{R}(X)=\min_{x\in \mathbb{R}}~\left \{  x+\mathcal{E}(X-x)\right \}
,\qquad \mathcal{D}(X)=\min_{x\in \mathbb{R}}~\left \{  \mathcal{E}(X-x)\right \}
,\qquad \mathcal{S}(X)=\underset{x\in \mathbb{R}}{\arg \min}~\left \{
\mathcal{E}(X-x)\right \}  .
\]
When $\mathcal{E}(X)=E\left[  L(X)\right]  $, for a loss function
$L:\mathbb{R}\rightarrow \mathbb{R}$, the risk quadrangle is called an
expectation quadrangle. \cite{wang2020risk} showed that an expectation
quadrangle has the convex level sets property, which is a necessary condition
for the notions of elicitability, identifiability, and backtestability. In the
expectation quadrangle case, the deviation $\mathcal{D}$ is called the
\emph{Bayes risk} of the statistic $\mathcal{S}$ in
\cite{FrongilloRafaelM2020Ecos}, and it was shown that the pair $(\mathcal{D}%
,\mathcal{S})$ which is called Bayes pair, is elicitable. The best known loss
function is the asymmetric $\ell_{1}$-norm and the corresponding Bayes pair is
the Expected Shortfall (ES) and the Value-at-Risk (VaR), denoted by $(\mathrm{ES}%
_{\tau},\mathrm{VaR}_{\tau})$\footnote{In the expectation quadrangle, when the loss function is the asymmetric $\ell_1$-norm loss function $L(x)=\frac{\tau}{1-\tau}x_{+}+x_{-}$, the Bayes pair $(\mathcal D,\mathcal S)$ is  $\mathcal D(X)=\mathrm{ES}_{\tau}(X)+E(X)$  and $\mathcal S(X)=\mathrm{VaR}_{\tau
}(X)$. That is in the case of $\ell_1$-norm loss function,  the Bayes risk of $\mathrm{VaR}_{\tau}$ is $\mathrm{ES}_{\tau}$   plus the expectation. In fact, in \cite{FrongilloRafaelM2020Ecos}, the Bayes pair is defined in a more general framework: for a loss function $L:\R\times\R\to\R$, the Bayes pair $(\mathcal{D}%
,\mathcal{S})$ is defined as  $\mathcal D(X)=\min_{x\in\R} L(X,x) $ and $S(X)={\rm argmin}_{x\in\R} L(X,x) $. Under this framework, taking the loss function as $L(y,x) = x+\frac1{1-\tau} (y-x)_+$, we have $(\mathcal{D},\mathcal{S})=(\mathrm{ES}_{\tau},\mathrm{VaR}_{\tau})$. }. The VaR at
confidence level $\tau \in(0,1)$ of a random variable $X$ with distribution
function $F$ is defined as $\mathrm{VaR}_{\tau}(X)=q_{\tau}(X)=\inf
\{x\in \mathbb{R}:F(x)\geq \tau \}$. The ES of $X$ at confidence level $\tau
\in(0,1)$ is defined as $\mathrm{ES}_{\tau}(X)=\frac{1}{1-\tau}\int_{\tau}%
^{1}q_{p}(X)\, \mathrm{d}p$. These two are the most popular risk
measures applied in practice and regulation.  Their properties, including asymptotic behaviors for
extreme risks and statistic inferences, have been widely explored in the
literature (see e.g., \citealp{acerbi2002coherence}; \citealp{frey2002var};  \citealp{scaillet2004nonparametric}; \citealp{so2006empirical}; \citealp{necir2010estimating};  \citealp{asimit2011asymptotics}; \citealp{hua2011second}; \citealp{chun2012conditional}).

Another
well-known loss function is the asymmetric $\ell_{2}$-norm
\begin{equation}
L_{\tau}(x)=\tau x_{+}^{2}+(1-\tau)x_{-}^{2},~~~\tau \in(0,1),\label{loss}%
\end{equation}
where $x_{+}=\max \{x,0\}$ and $x_{-}=-\min \{x,0\}$. Under the loss function $L_{\tau}$ defined by (\ref{loss}), the Bayes pair
$(\mathcal{D},\mathcal{S})$ is the variantile and the expectile, denoted by
$(var_{\tau},e_{\tau})$, where
\begin{equation}
{var}_{\tau}(X)=\min_{x\in \mathbb{R}}~E\left[  L_{\tau}(X-x)\right]  ,\qquad
e_{\tau}(X)=\underset{x\in \mathbb{R}}{\arg \min}~E\left[  L_{\tau}(X-x)\right]
.\label{bp}%
\end{equation}
By the definition of $e_{\tau}$, one immediately obtains that
\begin{equation}
{var}_{\tau}(X)=\tau{E}\left[  (X-e_{\tau})_{+}^{2}\right]  +(1-\tau
){E}\left[  (X-e_{\tau})_{-}^{2}\right]  .\label{var_tau}%
\end{equation}
%Thus, the expectile minimizes the expected loss of $X$ for the loss function $L_{\tau}$.
The expectile was first introduced in \cite{newey1987asymmetric}. Since \cite{bellini2014generalized} studied its
properties as a risk measure and \cite{ziegel2016coherence} showed that
expectiles are the only coherent risk measures that have the elicitability
property, expectile has
attracted much more attention in different areas, especially its asymptotic behavior for
extreme risks and statistic inference have been extensively studied; see, for example, \cite{mao2015asymptotic}, \cite{mao2015risk}, \cite{cai2016optimal}, \cite{holzmann2016expectile},
 \cite{bellini2017risk}, \cite{kratschmer2017statistical},  \cite{daouia2018estimation},  \cite{daouia2020tail}, \cite{girard2020nonparametric},  \cite{daouia2021expecthill}, and \cite{eberl2021expectile}. The variantile was first introduced in \cite{FrongilloRafaelM2020Ecos} as the Bayes risk of the expectile. {In view of the definitions in \eqref{bp}, the expectile is the constant
nearest to the risk $X$ in the sense of the asymmetric $\ell_{2}$-distance.
The variantile is the minimum expected loss of $X$ for the $\ell_{2}$-norm,
and thus it is the minimum $\ell_{2}$-distance. This shows that the variantile
is a deviation measure that generalizes the variance. Especially with the
asymmetric $\ell_{2}$-distance, it puts more weight to capture the right tail
behavior, which can be very useful for the study of the extreme risks. 
 \begin{comment}
 Another way to write the asymmetric $\ell_{2}$-norm is $\widetilde{L}_{\tau
}(x)=2x^{2}|1_{\{x\leq0\}}-\tau|$. Then $var_{\tau}$ and
$e_{\tau}$ can be calculated as
\[
{var}_{\tau}(X)=\min_{x\in \mathbb{R}}~E\left[  \widetilde{L}_{\tau
}(X-x)\right]  ,\qquad e_{\tau}(X)=\underset{x\in \mathbb{R}}{\arg \min
}~E\left[  \widetilde{L}_{\tau}(X-x)\right]  .
\]
From the above expressions, the expectile can be understood as a
compromise between the mean and a quantile and the variantile is  a compromise
between the variance and the variance of a superquantile.      
 \end{comment}
 As shown in \cite{FrongilloRafaelM2020Ecos}, the Bayes pair $(var_{\tau},e_{\tau})$ is elicitable and has other properties useful for the applications in such as  optimization and machine learning.}

However, little research has
been devoted to the variantile, except the study as the Bayes pair
$(var_{\tau},e_{\tau})$
%is in the
%expectation quadrangle with the loss function $L_{\tau}$ defined in
%(\ref{loss}). This pair was
in, e.g., \cite{FrongilloRafaelM2020Ecos} and \cite{wang2020risk}.
%The
%risk measure variantile was constructed using this method in
%\cite{FrongilloRafaelM2020Ecos}: it is the Bayes risk of the well-known
%expectile.
In this paper, we focus on the asymptotic behaviour and the estimation
of the modified variantile.  Next, we  explain the necessity to modify the variantile for measuring extreme risks. 

\begin{comment}
In the literature, the asymptotic behaviours of the ES and  expectile for extreme risks are usually studied with respect to VaR. The reason behind this is that    
\end{comment}
Many empirical studies have shown that asset returns in finance and large
losses in insurance exhibit power-like tails, which are also called heavy
tails; see, for example, \cite{loretan1994testing}, \cite{gabaix2003theory},
\cite{rachev2005fat} and \cite{gabaix2009power}. Moreover, regulators such as
Basel III have set the confidence level of VaR very close to $1$. 
%The extreme
%behaviors of financial risks have been the focus of the risk management.
%If our focus on extreme risks, then we are
Thus, we are interested in the behavior of a risk measure for heavy-tailed
risks when the confidence level $\tau$ is close to $1$. {Here, we present an
example to show that the direct application of variantile to extreme risks may  not  behave as other quantile based risk measures.} Consider that a risk follows a Pareto($\alpha,\theta$) distribution or a
Student's $t_{\alpha}$-distribution, where $\alpha$ represents the heavy tailedness of the distribution. Figure \ref{f1} shows that as
$\tau \rightarrow1$, the variantile  is converging to $0$. In
general, we can expect that the variantile for heavy-tailed risks is close to
$0$ when the confidence level $\tau$ is close to $1$. Mathematically, this can
be explained as the term $\tau{E}\left[  (X-e_{\tau})_{+}^{2}\right]  $ in
(\ref{var_tau}) vanishes due to $e_{\tau}$ approaching $\infty$ while
$(1-\tau){E}\left[  (X-e_{\tau})_{-}^{2}\right]  $ also vanishes because $1-\tau \rightarrow0$. This shows that the variantile is not a proper
measure for large losses because it fails to reflect the potentially large risk when
the confidence level is close to $1$. 
		\begin{figure}[]
		\centering
		\subfigure[Pareto($\alpha,\theta$) distribution]{
			\includegraphics[width=0.45\linewidth]{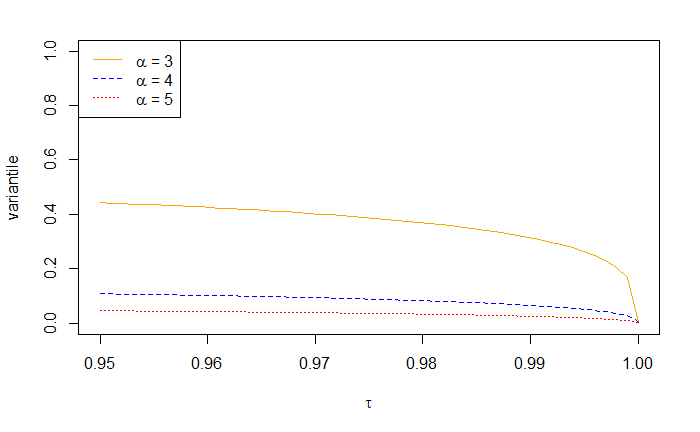}
			\label{s1}
		}
		\subfigure[Student's $t_\alpha$-distribution]{
		\includegraphics[width=0.45\linewidth]{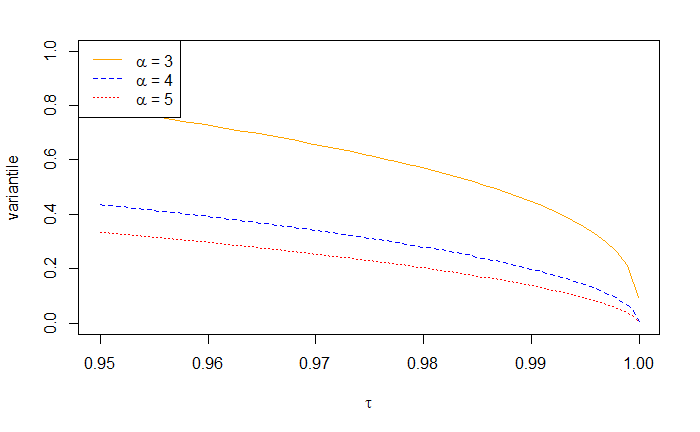}
			\label{s2}
		}
		\caption{The value of the variantile defined in \eqref{var_tau} with $X$ following a Pareto($\alpha,\theta$) or Student's $t_\alpha$-distribution.}
		\label{f1}
	\end{figure}
Hence, we propose the following adjustment to the variantile:%
\begin{equation}
dev_{\tau}(X):=\left(  \frac{\tau}{1-\tau}{E}\left[  (X-e_{\tau})_{+}%
^{2}\right]  +{E}\left[  (X-e_{\tau})_{-}^{2}\right]  \right)  ^{1/2}. \label{dev}
\end{equation}
The resulting risk measure is called the \emph{adjusted standard-deviatile},
or \emph{deviatile} for short. 
The adjustments made in the variantile are that the term $\tau{E}\left[
(X-e_{\tau})_{+}^{2}\right]  $ is scaled with $1-\tau$ and ${E}\left[
(X-e_{\tau})_{-}^{2}\right]  $ is not scaled. By doing so, neither term
vanishes as $\tau \rightarrow1$. Furthermore, we take the square root to ensure
that the resulting risk measure is comparable in scale with the risk itself. From the definition of${\ var}_{\tau}(X)$, it
is clear that
\[
dev_{\tau}(X)=\sqrt{\frac{{var}_{\tau}(X)}{1-\tau}}.%
\]
Thus, the square of the deviatile ${dev}_{\tau}^{2}(X)$ is still a Bayes risk of the
expectile with loss function $L_{\tau}^{\ast}(x)=\frac{\tau}{1-\tau}x_{+}^{2}+x_{-}^{2}$, $x\in \R$. This loss function under the above adjustments is also consistent with that of
the pair $(\mathrm{ES}_{\tau},\mathrm{VaR}_{\tau})$; see e.g. Example 2 of \cite{rockafellar2013fundamental}.

In this paper, we investigate the deviatile for extreme risks by deriving its
asymptotic expansions and exploring  efficient estimation methods. That is,
we are interested in the behavior and estimations of the deviatile for heavy-tailed risks when
$\tau$ is close to $1$. Asymptotic expansions provide an intuitive way to understand how the risk
measure behaves for extreme risks by expressing 
the risk measure as a determined function of VaR
when the confidence level is close to 1.  %For example, \cite{mao2015asymptotic} derived the expansions
%for the expectile, and \cite{mao2012second} and \cite{tang2012haezendonck} investigated the Haezendonck-Goovaerts risk measure. 
In this paper, as a
first step, we derive the first- and second-order asymptotic expansions for
the deviatile.\footnote{Noting the relationship between variantile and deviatile, all the results in this paper also serve as  asymptotic expansions of variantile.}

Inspired by \cite{daouia2018estimation} and \cite{zhao2021estimation}, the first-order asymptotic expansion of the deviatile can also be used to estimate
the deviatile at high confidence levels. This estimation method has been
shown to be easy to implement and exhibit good performance. In this paper, we
propose two estimators for the deviatile. The first is the
intermediate-level estimator based on the first-order asymptotic expansion, which is
useful when $\tau$ is close to $1$ but is still within the data range. The second
is the extreme-level estimator obtained by extrapolating the intermediate-level
estimator to a higher confidence level, which can even be outside the data
range. {We first investigate the asymptotic normality of these two estimators when the sample is independent and identically distributed (i.i.d.)  using the techniques from extreme value theory. Then we relax the condition to allow the sample to have the serial dependence. This is because, in
practice, the data may not be completely independent. For example the asset
returns have been shown to be sequentially dependent, which often is modeled
by GARCH-processes; see, e.g. \cite{mcneil2000estimation} and \cite{chavez2014extreme}. The class of $\beta$-mixing time series
includes many important models used in practice, such as ARMA, ARCH and
GARCH models. Again with the help of extreme value theory, we prove the asymptotic normality of both estimators when the sample is a $\beta$-mixing series.} Simulations are
provided to examine the performance of the two estimators. Moreover, real data  analysis is carried out to compare the deviatile with VaR, ES, and expectile. 

The remainder of the paper is organized as follows. Section \ref{Prlm} reviews heavy-tailed distributions and the asymptotic expansions of
the expectile. In Section \ref{AE}, we derive the asymptotic expansions of the deviatile. In Section \ref{estimation}, estimation methods based on the expansions of the
deviatile are studied. Section \ref{simulation} presents the simulation study, and Section
\ref{real} conducts the real data study of the deviatile. Lastly, Section \ref{conclusion}
concludes the paper.

	%For example, the common VaR, Expectile and CVaR are all risk measures with good applications. At the same time, different risk measures have their own advantages and disadvantages. Therefore, it is reasonable to use different risk measures in different situations and different problems. Not only that, when traditional risk measures are not enough, we also need to define new risk measures.

	\section{Preliminaries on RV and 2RV}\label{Prlm}
	
		Throughout the paper,  $f(x)\sim g(x)$ as $x\to x_0$ is used to represent that
$\lim_{x\to x_0} f(x)/g(x)=1$, and $f(x)=o(1)$ as $x\to x_0$ represents that $\lim_{x\rightarrow
	x_0}f(x)= 0$.
%	\subsection{}
	We first present the definitions of regular variation and the
        second-order regular variation, which are standard assumptions in modeling heavy-tailed risks in extreme value theory. In particular, the second-order regular variation is essential in establishing the convergence rate of the estimation; see \cite{de2007extreme}.
	\begin{definition}
		A Lebesgue measurable function $f:\mathbb{R}\rightarrow \mathbb{R}$ that is
eventually positive is said to be regularly varying (RV) at $\infty$ with
index $\gamma \in \mathbb{R}$ such that
\[
\lim_{t\rightarrow \infty}\frac{f(tx)}{f(t)}=x^{\gamma},\quad x>0.
\]
This is denoted by $f(\cdot)\in \mathrm{RV}_{\gamma}$, and $\gamma$ is called
the tail index. When $\gamma=0$, $f$ is called a slowly varying function.
	
	\end{definition}
	
A risk $X$ with a distribution function $F=1-\overline{F}$ is said to be a
RV random variable/risk if $\overline{F}\in \mathrm{RV}%
_{-\alpha}$ for some $\alpha>0$.
	
	%	The following concept of second-order regular variation is essential in studying the asymptotic normality.
	\begin{definition}
		A Lebesgue measurable function $f: \mathbb{R} \rightarrow \mathbb{R}$ that is eventually positive is said to
		be  second-order regularly varying with the first-order parameter $\gamma \in \mathbb{R}$ and the second-order
		parameter $\rho \leqslant 0$, denoted by $f \in 2 \mathrm{RV}_{\gamma, \rho}$, if there exists some ultimately positive or negative
		function $A(t)$ with $A(t) \rightarrow 0$ as $t \rightarrow \infty$ such that
		$$
		\lim _{t \rightarrow \infty} \frac{f(t x) / f(t)-x^{\gamma}}{A(t)}=x^{\gamma} \int_{1}^{x} u^{\rho-1} d u, \quad x>0,
		$$
		where $\int_{1}^{x} u^{\rho-1} d u$ reads as $\log x$ when $\rho=0$. Here, $A(t)$ is referred to as an auxiliary function of $f$.
	\end{definition}

	The tail quantile function $U$ of $X$ is the left-continuous inverse of
$1/\overline{F}$; that is, for $t>0$,%
\[
U(t)=\left(  \frac{1}{\overline{F}}\right)  ^{\leftarrow}(t)=F^{\leftarrow
}\left(  1-\frac{1}{t}\right)  .
\]
It is easy to check that $\overline{F}\in \mathrm{RV}_{-1/\gamma}$ for
$\gamma>0$ if and only if $U\in \mathrm{RV}_{\gamma}$. Furthermore, Theorem 2.3.9
of \cite{de2007extreme} states an equivalence in the second-order condition
between $\overline{F}(\cdot)$ and $U(\cdot)$ that for $\gamma>0$ and
$\rho \leqslant0,\overline{F}\in2\mathrm{RV}_{-1/\gamma,\rho/\gamma}$ with an
auxiliary function $\gamma^{-2} A(1/\overline{F}(\cdot))$ if and only if
$U(\cdot)\in2\mathrm{RV}_{\gamma,\rho}$ with an auxiliary function $A(\cdot)$.

	\section{Asymptotic Expansions of the Deviatile} \label{AE}
	Based on the asymptotic expansions of the expectile and the
        quantile, in this section, we derive the first- and
        second-order asymptotic expansions for the  deviatile. All the
        proofs are relegated to the Appendix.

\begin{theorem}
\label{1st}Assume that $U(t)\in \mathrm{RV}_{\gamma}$ with $0<\gamma<1/2$. Then as $\tau \uparrow1$,
\[
{dev}_{\tau}(X)\sim \beta_{\gamma}q_{\tau},
\]
where $\beta_{\gamma}=\frac{(\gamma^{-1}-1)^{-\gamma}}{\sqrt{1-2\gamma}}$.
%%with$$a(\tau)=(1-2\gamma)(\frac{(\gamma^{-1}-1)^{2\gamma}}{q_{\tau}^2}{E}(X)^2-\frac{2\gamma(\gamma^{-1}-1)^{2\gamma}({E}X)^2}{q_{\tau}^2}-\frac{2(\gamma^{-1}-1)^{\gamma}{E}X}{q_{\tau}}  )$$

\end{theorem}

In the proof of Theorem \ref{1st}, the relation (\ref{1stavar}) shows that ${dev}_{\tau
}^2(X)$  is asymptotically linearly
proportional to the square of $q_{\tau}$ as 	$\tau \to 1$ while other risk
measures such as the expectile are linearly proportional to
$q_{\tau}$ (Proportion \ref{p1}). {This again verifies the necessity of
introducing the adjusted
version of the variantile so that the deviatile is linearly proportional to $q_{\tau}$, which makes it comparable to other quantile based risk measures.}
	
{
 A by-product from the proof of Theorem \ref{1st} is that we can expand the
deviatile as a linear function of the expectile as follows. 

\begin{corollary}
\label{cor1}Under the conditions of Theorem \ref{1st}, as $\tau \rightarrow1$,
we have%
\[
dev_{\tau}(X)\sim(1-2\gamma)^{-1/2}e_{\tau}.
\]

\end{corollary}

Combining Theorem \ref{1st}, Corollary \ref{cor1} and Proposition \ref{p1} in
Appendix, we obtain the following rankings of the deviatile, quantile and
expectile of a heavy tailed $X$ as $\tau \rightarrow1$,
\[
e_{\tau}\lesssim q_{\tau}(X),~~e_{\tau}\lesssim dev_{\tau}(X)~~
\text{for all }0<\gamma<1/2,
\]
and
\[
%\left \{
%\begin{array}
%[c]{lll}%
q_{\tau}(X)\lesssim dev_{\tau}(X)~ {\rm for}~\gamma^{\ast}\leq \gamma<1/2,~~~~
q_{\tau}(X)\gtrsim dev_{\tau}(X)~ {\rm for}~ 0<\gamma \leq \gamma^{\ast},
%\end{array}
%\right.
\]
where  
$f(x)\lesssim(\gtrsim)g(x)$ as $x\rightarrow x_{0}$ means that $\lim_{x\rightarrow
x_{0}}f(x)/g(x)\leq(\geq)1$. Here $\gamma^{\ast}$ is the unique solution in
$(0,1/2)$ of the equation
\[
(\gamma^{-1}-1)^{-\gamma}=\sqrt{1-2\gamma}.
\]
The numerical value of $\gamma^{\ast}$ is approximately $0.2135$. The
existence of $\gamma^{\ast}$ follows from the standard algebraic manipulations
such as finding derivative and proving the monotonicity of function $(\gamma^{-1}-1)^{-\gamma}-\sqrt{1-2\gamma}$. Such comparisons illustrate
that at the same confidence level, the value of deviatile is the largest among
the expectile and quantile when the risk has a heavier tail (with the tail
index greater than $\gamma^{\ast}$).%
}

Next, we present the second-order asymptotic expansions of the deviatile.

\begin{theorem}
\label{2ndexp}Assume that $U\in2\mathrm{RV}_{\gamma,\rho}$ with $0<\gamma
<1/2$, $\rho \leq0$ and auxiliary function $A(t)$. Then as $\tau \uparrow1$,
\[
dev_{\tau}(X)=\beta_{\gamma}q_{\tau}\left(  1+w(\tau)\right)  ,
\]
where
\[
w(\tau)=\xi_{1}q_{\tau}^{-1}(1+o(1))+\xi_{2}A\left(  (1-\tau)^{-1}\right)(1+o(1)),
\]
\[
\xi_{1}=(2\gamma-1)\left(  \gamma^{-1}-1\right)  ^{\gamma}{E}(X)
~~{\rm and}~~
\xi_{2}=\frac{\left(  \gamma^{-1}-1\right)  ^{-\rho}\left(  2-3\gamma
-\rho \right)  }{\left(  1-\rho-2\gamma \right)  \left(  1-\rho-\gamma \right)
}+\frac{\left(  \gamma^{-1}-1\right)  ^{-\rho}-1}{\rho}.
\]
%check:
%\begin{align*}
%\xi_{2}  & =2\gamma \left(  -\frac{\left(  \gamma^{-1}-1\right)  ^{-\rho}%
%}{\gamma(1-\rho-\gamma)}+2\left(  \frac{\left(  \gamma^{-1}-1\right)  ^{-\rho
%}}{1-\rho-\gamma}+\frac{\left(  \gamma^{-1}-1\right)  ^{-\rho}-1}{\rho
%}\right)  +\frac{(2-\rho-3\gamma)\left(  \gamma^{-1}-1\right)  ^{-\rho}%
%}{\gamma \left(  1-\rho-2\gamma \right)  \left(  1-\rho-\gamma \right)  }\right)
%+2(1-2\gamma)\left(  \frac{\left(  \gamma^{-1}-1\right)  ^{-\rho}}%
%{1-\rho-\gamma}+\frac{\left(  \gamma^{-1}-1\right)  ^{-\rho}-1}{\rho}\right)
%\\
%& =2\gamma \left(  \frac{(1-\gamma)\left(  \gamma^{-1}-1\right)  ^{-\rho}%
%}{\gamma \left(  1-\rho-2\gamma \right)  \left(  1-\rho-\gamma \right)  }%
%+\frac{2\left(  \gamma^{-1}-1\right)  ^{-\rho}}{1-\rho-\gamma}+2\frac{\left(
%\gamma^{-1}-1\right)  ^{-\rho}-1}{\rho}\right)  +2(1-2\gamma)\left(
%\frac{\left(  \gamma^{-1}-1\right)  ^{-\rho}}{1-\rho-\gamma}+\frac{\left(
%\gamma^{-1}-1\right)  ^{-\rho}-1}{\rho}\right)  \\
%& =\frac{2\left(  \gamma^{-1}-1\right)  ^{-\rho}\left(  2-3\gamma-\rho \right)
%}{\left(  1-\rho-2\gamma \right)  \left(  1-\rho-\gamma \right)  }%
%+2\frac{\left(  \gamma^{-1}-1\right)  ^{-\rho}-1}{\rho}%
%\end{align*}

\end{theorem}

\bigskip
Next, we provide some examples to demonstrate the accuracy of the first-
and second-order expansions of the deviatile given in Theorems
\ref{1st} and \ref{2ndexp}.

A Pareto distribution with parameters $\alpha>0$ and $\theta>0 $, denoted by Pareto($\alpha,\theta$), has a distribution function of
	$$
	F(x)=1-\left(\frac{\theta}{x+\theta}\right)^{\alpha}, \quad x>0.
	$$
It is easy to verify that the tail quantile function of the Pareto distribution satisfies that
$U(t)\in2\mathrm{RV}_{\gamma,-\gamma}$ with auxiliary function $A(t)=\gamma
t^{-\gamma}$, $\gamma=1/\alpha$ and $\rho=-1/\alpha$.
{In Figure \ref{fig1} (a, c), we plot the true values of the deviatile for a
Pareto($\alpha,1$) distribution when $\alpha=2.2$ and $3$ and its first- and
second-order expansions for $\tau$ ranging from $0.95$ to $0.999$. In Figure
\ref{fig1} (b, d), we plot the relative error of each expansion.  The true values of the expectile for the Pareto distribution are calculated numerically with R and  then these values are plugged in to \eqref{dev} to obtain the true values of the deviatile.} From these plots, we can see that the second-order asymptotic expansion improves the first-order asymptotic expansion, especially in the
heavier-tailed case. 
	
		\begin{figure}[]
		\centering
		\subfigure[$\alpha=2.2$]{
			\includegraphics[width=0.45\linewidth]{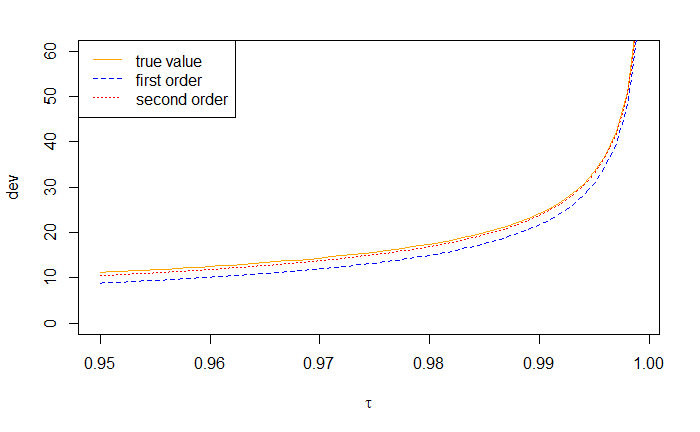}
			\label{sub1}
		}
		\subfigure[$\alpha=2.2$, relative error]{
		\includegraphics[width=0.45\linewidth]{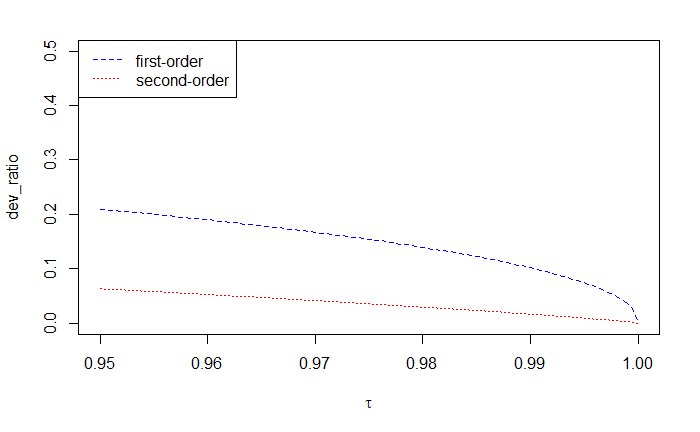}
			\label{sub2}
		}
		\quad    %用 \quad 来换行
		\subfigure[$\alpha=3$]{
		\includegraphics[width=0.45\linewidth]{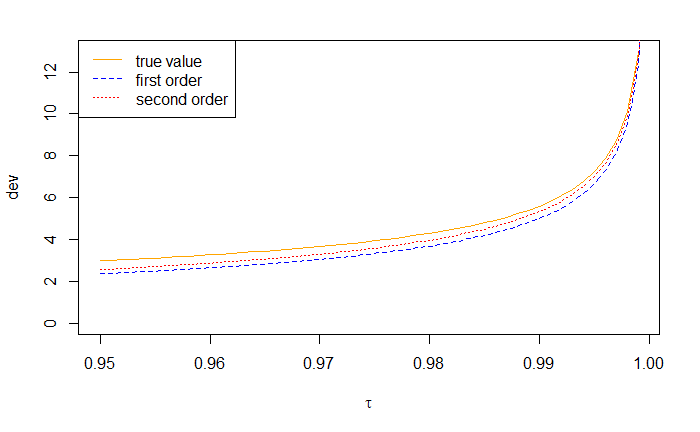}
			\label{sub3}
		}
		\subfigure[$\alpha=3$, relative error]{
			\includegraphics[width=0.45\linewidth]{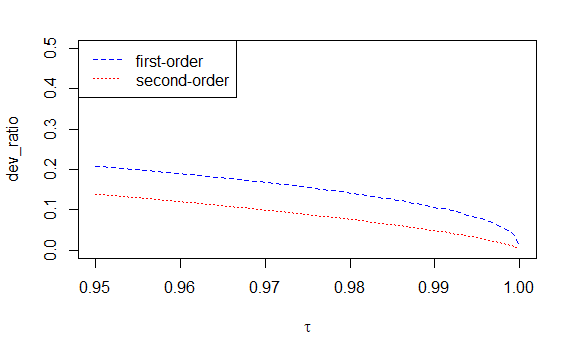}
			\label{sub4}
		}
		\caption{Asymptotic expansions of the deviatile for the Pareto($\alpha,1$) distribution }
		\label{fig1}
	\end{figure}

The Student's $t_{\alpha}$-distribution has a density function of
\begin{equation}
 f(x)=\frac{\Gamma((\alpha+1)/2)}{\sqrt{\alpha \pi}\Gamma(\alpha/2)}\left(
1+\frac{x^{2}}{\alpha}\right)  ^{-(\alpha+1)/2},\qquad x\in \mathbb{R}.   \label{st}
\end{equation}
One can verify that $\overline{F}(t)\in2\mathrm{RV}_{-\alpha,-2},$ and
$U(t)\in2\mathrm{RV}_{\gamma,-2\gamma},$ with $\gamma=1/\alpha,\rho=-2/\alpha$,
\[
A(t)\sim \frac{\alpha+1}{\alpha+2}\left(  c_{\alpha}t\right)  ^{-2/\alpha
}\quad{\rm and}\quad c_{\alpha}=\frac{2\Gamma((\alpha+1)/2)\alpha^{(\alpha-1)/2}}%
{\sqrt{\alpha \pi}\Gamma(\alpha/2)}.
\]
{In Figure \ref{fig2} (a, c), we plot the true values and the first- and second-order
expansions of the deviatile for a Student's $t_{\alpha}$-distribution
when $\alpha=2.2$ and $3$.  In Figure \ref{fig2} (b, d), we plot the relative error of the each expansion. The true values of the expectile for the Student's $t_{\alpha}$-distribution are calculated numerically with R and the true values of deviatile are calculated accordingly.} Both first- and second-order expansions are better
approximations of the deviatile in the heavier-tailed case. 
	
		\begin{figure}[]
		\centering
		\subfigure[$\alpha=2.2$]{
			\includegraphics[width=0.45\linewidth]{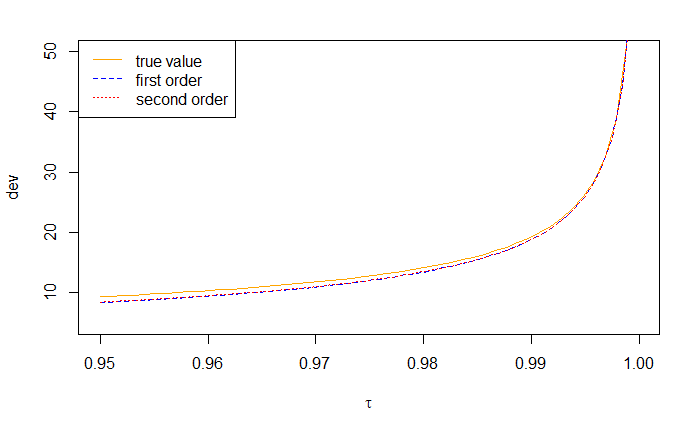}
			\label{sub5}
		}
	\subfigure[$\alpha=2.2$, relative error]{
		\includegraphics[width=0.45\linewidth]{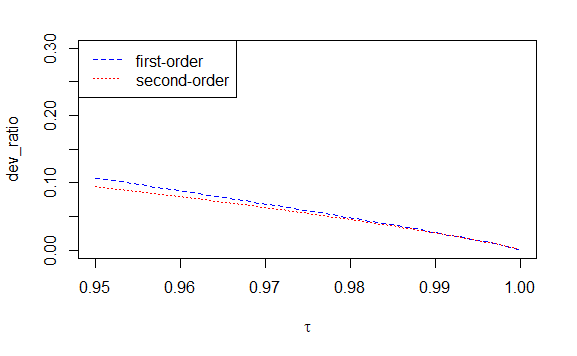}
			\label{sub6}
		}
	%	\quad    %用 \quad 来换行
		\subfigure[$\alpha=3$]{
		\includegraphics[width=0.45\linewidth]{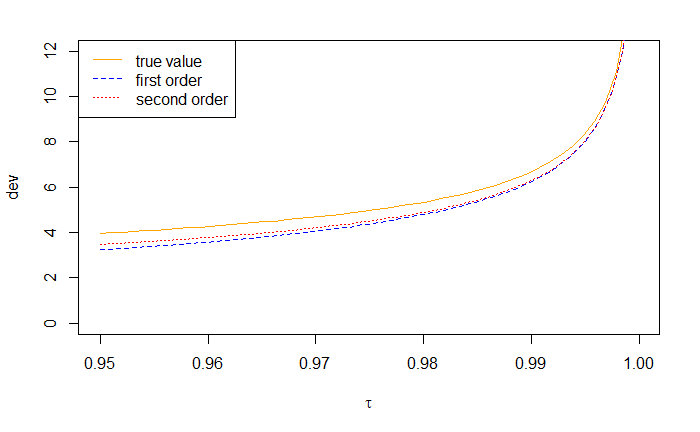}
			\label{sub7}
		}
		\subfigure[$\alpha=3$, relative error]{
			\includegraphics[width=0.45\linewidth]{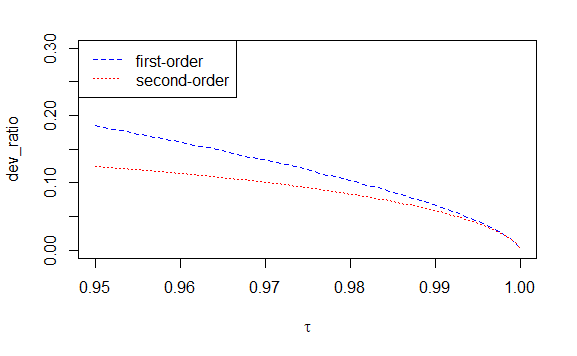}
			\label{sub8}
		}
		\caption{Asymptotic expansions of the deviatile for
                  the Student's $t_\alpha$-distribution }
		\label{fig2}
	\end{figure}

\section{Estimation} \label{estimation}

In this section, we propose estimators of the deviatile at the intermediate
level $\tau_{n}$, where $\tau_{n}\rightarrow1$ and  $n\left(  1-\tau
_{n}\right)  \rightarrow \infty$ as $n\rightarrow \infty$, and the extreme
levels $p_{n}$, where $p_{n}\rightarrow1$ and $n\left(  1-p_{n}\right)
\rightarrow c$ for some $c>0$. The asymptotic normality of both
estimators are investigated when the sample is i.i.d. and is a $\beta$-mixing series, respectively. All the proofs are relegated to the Appendix.

\subsection{Independence}
Suppose that $X,X_{1},...,X_{n}$ are i.i.d. random
variables with distribution function $F$. The order statistics are denoted by %
$
X_{1,n}\leq X_{2,n}\leq \cdots \leq X_{n,n}
$.

\subsubsection{Intermediate level}

Inspired by the asymptotic expansions in Theorem \ref{1st}, we propose the
intermediate-level estimator of the deviatile as
\[
\widehat{dev}_{\tau_{n}}=\widehat{\beta}_{\gamma}\widehat{q}_{\tau_{n}},
\]
where %
\begin{equation}
 \widehat{\beta}_{\gamma}=\frac{(\widehat{\gamma}^{-1}-1)^{-\widehat{\gamma}}%
}{\sqrt{1-2\widehat{\gamma}}},
\label{beta}   
\end{equation}
$\widehat{q}_{\tau_{n}}$ 
is the estimator of the quantile $q_{\tau_{n}}$ given by 
\[
\widehat{q}_{\tau_{n}}=X_{n-\left \lfloor n\left(  1-\tau_{n}\right)
\right \rfloor ,n},%
\] and $\left \lfloor
\cdot \right \rfloor $ denotes the floor function. In \eqref{beta}, $\widehat{\gamma}$ is an estimator of the tail index $\gamma$, and the
common choices are the Hill estimator \citep{hill1975simple}, the moment estimator \citep{dekkers1989moment} and
the maximum likelihood estimator \citep{drees2004maximum}. This estimation method is
motivated by \cite{daouia2018estimation}.

Next, we explore the asymptotic normality for $\widehat{dev}_{\tau_{n}}$.

\begin{theorem}
\label{int}Assume that $U\in2\mathrm{RV}_{\gamma,\rho}$ with $0<\gamma
<1/2$, $\rho \leq0$ and auxiliary function $A(t)$. Further assume that
$\tau_{n}\rightarrow1$, $n\left(  1-\tau_{n}\right)  \rightarrow \infty$ as
$n\rightarrow \infty$, and
\begin{equation}
\sqrt{n\left(  1-\tau_{n}\right)  }\left(  \widehat{\gamma}-\gamma
,\frac{\widehat{q}_{\tau_{n}}}{q_{\tau_{n}}}-1\right)  \overset{d}%
{\longrightarrow}(\Gamma,\Theta).\label{limit}%
\end{equation}
If $\sqrt{n\left(  1-\tau_{n}\right)  }q_{\tau_{n}}^{-1}\rightarrow \lambda
_{1}\in \mathbb{R}$ and $\sqrt{n\left(  1-\tau_{n}\right)  }A\left(  \left(
1-\tau_{n}\right)  ^{-1}\right)  \rightarrow \lambda_{2}\in \mathbb{R}$, then we
have that as $n\rightarrow \infty$,
\[
\sqrt{n\left(  1-\tau_{n}\right)  }\left(  \frac{\widehat{dev}_{\tau_{n}}%
}{dev_{\tau_{n}}}-1\right)  \overset{d}{\longrightarrow}m(\gamma)\Gamma
+\Theta-\left(  \lambda_{1}\xi_{1}+\lambda_{2}\xi_{2}\right)  ,
\]
where
\[
m(\gamma)=-\log(\gamma^{-1}-1)+\frac{1}{1-\gamma}+\frac{1}{1-2\gamma}.
\]

\end{theorem}

\bigskip

Consider the Hill estimator for $\gamma$
\[
\widehat{\gamma}_{H}=\frac{1}{k}\sum_{i=0}^{k-1}\log X_{n-i,n}-\log X_{n-k,n}.
\]
By Theorem 3.2.5 in \cite{de2007extreme}, under some conditions, when
$\widehat{\gamma}=\widehat{\gamma}_{H}$ in (\ref{limit}), the limiting
distribution $\Gamma$ satisfies that
\[
\Gamma \overset{d}{=}\mathcal{N}\left(  \frac{\lambda_{2}}{1-\rho},\gamma
^{2}\right)  ,
\]
and\ by Theorem 2.4.1 in \cite{de2007extreme},
\[
\Theta \overset{d}{=}\mathcal{N}\left(  0,\gamma^{2}\right)  .
\]
Furthermore, Lemma 3.2.3 of \cite{de2007extreme} shows that both Gaussian limiting
distributions are independent. 

When $\widehat{\gamma}=\widehat{\gamma}_{H}$, we denote the estimator of the deviatile as
$\widehat{dev}_{\tau_{n}}(H)$. Combining the above results, we obtain the
asymptotic normality of $\widehat{dev}_{\tau_{n}}(H)$ in the following corollary.

\begin{corollary}\label{cor1}
Assume that  $U(t)\in2\mathrm{RV}%
_{\gamma,\rho}$ with $0<\gamma<1/2$ and $\rho \leq0$ with
auxiliary function $A(\cdot)$. Further assume that $\tau_{n}\rightarrow1$,
$n\left(  1-\tau_{n}\right)  \rightarrow \infty$ as $n\rightarrow \infty$. If
$\sqrt{n\left(  1-\tau_{n}\right)  }q_{\tau_{n}}^{-1}\rightarrow \lambda_{1}%
\in \mathbb{R}$ and $\sqrt{n\left(  1-\tau_{n}\right)  }A((1-\tau_{n}%
)^{-1})\rightarrow \lambda_{2}\in \mathbb{R}$, then as
$n\rightarrow \infty$,
\[
\sqrt{n\left(  1-\tau_{n}\right)  }\left(  \frac{\widehat{dev}_{\tau_{n}}%
(H)}{dev_{\tau_{n}}(H)}-1\right)  \overset{d}{\longrightarrow}\mathcal{N}%
\left(  \frac{m(\gamma)}{1-\rho}\lambda_{2}-\left(  \lambda_{1}\xi_{1}%
+\lambda_{2}\xi_{2}\right)  ,v(\gamma)\right)  ,
\]
where $m(\gamma)$ given in Theorem \ref{int} and $v(\gamma)=\gamma
^{2}[1+m(\gamma)^{2}]$.
\end{corollary}
{
\begin{remark}
We examined the quality of asymptotic approximations in Corollary \ref{cor1}
by plotting the sample quantiles versus the theoretical quantiles of the
limiting distribution. We observed that the sample variance of the estimators
are not quite close to the variance obtained via a naive use of the
theoretical Gaussian approximation. This issue has been observed in Section A.3 of
\cite{daouia2018estimation} and it is partially due to the estimation
uncertainty. Possible solutions of this issue have been proposed in
\cite{padoan2022joint}. Interested readers are referred to this paper for more details.
\end{remark}
} %

\subsubsection{Extreme level}

In this section, we propose an estimator of the deviatile at extreme
level and study its asymptotic normality. The extreme level means the
confidence level $p_{n}\uparrow1$ satisfying $n\left(  1-p_{n}\right)
\rightarrow c$ for some $c>0$ as $n\rightarrow \infty$. The basic idea is to
extrapolate the intermediate estimation to the very extreme level, which is
motivated by the methods in \cite{weissman1978estimation} and \cite{daouia2018estimation}.

Let $\left \{  p_{n},n\in \mathbb{N}\right \}  $ and $\left \{  \tau_{n}%
,n\in \mathbb{N}\right \}  $ be two sequences of levels satisfying that as
$n\rightarrow \infty$, $p_{n}\rightarrow1$, $n\left(  1-p_{n}\right)
\rightarrow c$ for some constant $c\in \mathbb{R}_{+}$, and $\tau
_{n}\rightarrow1$, $n\left(  1-\tau_{n}\right)  \rightarrow \infty$. Then, by
the assumption that the tail quantile function $U$ of a random variable $X\sim
F$ satisfies $U\in \mathrm{RV}_{\gamma}$, we have
\[
\frac{q_{p_{n}}}{q_{\tau_{n}}}=\frac{U\left(  \left(  1-p_{n}\right)
^{-1}\right)  }{U\left(  \left(  1-\tau_{n}\right)  ^{-1}\right)  }\sim \left(
\frac{1-p_{n}}{1-\tau_{n}}\right)  ^{-\gamma},\qquad n\rightarrow \infty.
\]
It then follows from Theorem \ref{2ndexp} that
\[
\frac{dev_{p_{n}}(X)}{dev_{\tau_{n}}(X)}\sim \frac{\beta_{\gamma}q_{p_{n}}%
}{\beta_{\gamma}q_{\tau_{n}}}\sim \left(  \frac{1-p_{n}}{1-\tau_{n}}\right)
^{-\gamma},\qquad n\rightarrow \infty.
\]
Based on the above relationship, we propose the following estimator of
$dev_{p_{n}}(X)$ at the extreme level
\[
\widehat{dev}_{p_{n}}^{\ast}=\left(  \frac{1-p_{n}}{1-{\tau_{n}}}\right)
^{-\widehat{\gamma}}\widehat{dev}_{\tau_{n}}.
\]

\begin{theorem}
\label{extreme}Assume that $U\in2\mathrm{RV}_{\gamma,\rho}$ with $0<\gamma
<1/2$, $\rho \leq0$ and auxiliary function $A(t)$.  As $n\rightarrow \infty$, let $\tau_{n},p_{n}\rightarrow1$ with $n\left(
1-\tau_{n}\right)  \rightarrow \infty$ and $n\left(  1-p_{n}\right)
\rightarrow c<\infty$, and further assume
\[
\sqrt{n\left(  1-\tau_{n}\right)  }\left(  \frac{\widehat{dev}_{\tau_{n}}%
}{dev_{\tau_{n}}}-1\right)  \overset{d}{\longrightarrow}\Delta \qquad
\text{and}\qquad \sqrt{n\left(  1-\tau_{n}\right)  }(\widehat{\gamma}%
-\gamma)\overset{d}{\longrightarrow}\Gamma.
\]
If $\sqrt{n\left(  1-\tau_{n}\right)  }q_{\tau_{n}}^{-1}\rightarrow \lambda
_{1}\in \mathbb{R}$, $\sqrt{n\left(  1-\tau_{n}\right)  }A((1-\tau_{n}%
)^{-1})\rightarrow \lambda_{2}\in \mathbb{R}$ and $\sqrt{n\left(  1-\tau
_{n}\right)  }/\log \left[  \left(  1-\tau_{n}\right)  /\left(  1-p_{n}\right)
\right]  \rightarrow \infty$, then as
$n\rightarrow \infty$,
\[
\frac{\sqrt{n\left(  1-\tau_{n}\right)  }}{\log \left[  \left(  1-\tau
_{n}\right)  /\left(  1-p_{n}\right)  \right]  }\left(  \frac{\widehat
{dev}_{p_{n}}^{\ast}}{dev_{p_{n}}}-1\right)  \overset{d}{\longrightarrow
}\Gamma.
\]

\end{theorem}

\bigskip

When $\gamma$ is estimated by the Hill estimator, we denote the extreme-level
estimator by $\widehat{dev}_{p_{n}}^{\ast}(H)$. Note that the limiting
distribution $\Gamma$ of the Hill estimator is $\mathcal{N}\left(
\lambda/(1-\rho),\gamma^{2}\right)  $. Substituting this into Theorem
\ref{extreme} yields the following result.

\begin{corollary}
Assume that $U(t)\in2\mathrm{RV}_{\gamma,\rho}$ with $0<\gamma<1$, $\rho \leq0$
and auxiliary function $A(t)$. As $n\rightarrow \infty$, let $\tau_{n},p_{n}\rightarrow1$ with $n\left(
1-\tau_{n}\right)  \rightarrow \infty$ and $n\left(  1-p_{n}\right)
\rightarrow c<\infty$, and further assume  that
\[
\sqrt{n\left(  1-\tau_{n}\right)  }\left(  \frac{\widehat{dev}_{\tau_{n}}%
}{dev_{\tau_{n}}}-1\right)  \overset{d}{\longrightarrow}\Delta.
\]
If $\sqrt{n\left(  1-\tau_{n}\right)  }q_{\tau_{n}}^{-1}\rightarrow \lambda
_{1}\in \mathbb{R}$, $\sqrt{n\left(  1-\tau_{n}\right)  }A((1-\tau_{n}%
)^{-1})\rightarrow \lambda_{2}\in \mathbb{R}$ and $\sqrt{n\left(  1-\tau
_{n}\right)  }/\log \left[  \left(  1-\tau_{n}\right)  /\left(  1-p_{n}\right)
\right]  \rightarrow \infty$, then as $n\rightarrow \infty$,
\[
\frac{\sqrt{n\left(  1-\tau_{n}\right)  }}{\log \left[  \left(  1-\tau
_{n}\right)  /\left(  1-p_{n}\right)  \right]  }\left(  \frac{\widehat
{dev}_{p_{n}}^{\ast}(H)}{dev_{p_{n}}(H)}-1\right)  \overset{d}{\longrightarrow
}\mathcal{N}\left(  \lambda/(1-\rho),\gamma^{2}\right)  .
\]

\end{corollary}

{
\subsection{Serial dependence}

In the last subsection, we developed the asymptotic normality of the
estimators for deviatile when  $X, X_{1},...,X_{n}$ are i.i.d.. In this subsection, we consider the sample to have serial dependence, that is it is from a $\beta$-mixing time series. Its formal definition is given as follows. Suppose
$\{U_{i}\}_{i\in \mathbb{N}}$ is a sequence of uniformly distributed random
variables. It is $\beta$-mixing (or absolute regular) if
\[
\beta(k)=\sup_{l\in \mathbb{N}}E\left[  \sup_{A\in \mathcal{B}_{l+k+1}^{\infty}%
}\left \vert \Pr \left(  A|\mathcal{B}_{1}^{l}\right)  -\Pr(A)\right \vert
\right]  \rightarrow0,
\]
as $k\rightarrow \infty$, where $\mathcal{B}_{1}^{l}$ and $\mathcal{B}%
_{l+k+1}^{\infty}$ denote the $\sigma$-fields generated by $\{U_{i}\}_{1\leq
i\leq l}$ and $\{U_{i}\}_{l+k+1\leq i<\infty}$. In this section we investigate
the asymptotic normality of the estimators for deviatile when the sample is
$\beta$-mixing.

Define the tail empirical function
\[
Q_{n}(t)=X_{n-[k_{n}t],n},\qquad0\leq t\leq1.
\]
It is very useful in the estimations of high quantiles and the tail index
$\gamma$. For example,
\[
\widehat{q}_{\tau_{n}}=Q_{n}(1)
\]
with $k_{n}=n(1-\tau_{n})$. Many estimators of $\gamma$ can be represented as a
tail functional $T(Q_{n})$. For example the Hill estimator has the functional
\[
T_{H}(z)=\int_{0}^{1}\log \frac{z(t)}{z(1)}dt.
\]
For details, see \cite{drees1998smooth}. The asymptotic behavior of $Q_{n}%
(t)$, as as well as $T(Q_{n})$ and extreme quantiles, for the $\beta$-mixing
time series have been well studied \cite{drees2000weighted},
\cite{drees2003extreme}, \cite{de2016adapting}, and \cite{chavez2018extreme}. These results enable us
to obtain the asymptotic normality of $\widehat{dev}_{\tau_{n}}$ when
$X_{i}$'s are $\beta$-mixing. Before we present the theorem, we introduce a
few regularity conditions on the $\beta$-mixing coefficient.

{\textbf{Regularity condition of $\beta$ (C$_{R}$)}}: There exist $\varepsilon>0$, a
function $r$ and a sequence $l_{n}$ such that, as $n\rightarrow \infty$,

\begin{description}
\item[(C1)] $\frac{\beta(l_{n})}{l_{n}}n+l_{n}\frac{\log^{2}k}{\sqrt{k}%
}\rightarrow0;$

\item[(C2)] $\frac{n}{l_{n}k}cov\left(  \sum_{i=1}^{l_{n}}1_{\left \{
X_{i}>F^{\leftarrow}(1-kx/n)\right \}  },\sum_{i=1}^{l_{n}}1_{\left \{
X_{i}>F^{\leftarrow}(1-ky/n)\right \}  }\right)  \rightarrow r(x,y)$, for
any $0\leq x,y\leq1+\varepsilon$;

\item[(C3)] For some constant $C:$%
\[
\frac{n}{l_{n}k}E\left[  \left(  \sum_{i=1}^{l_{n}}1_{\left \{  F^{\leftarrow
}(1-ky/n)<X_{i}\leq F^{\leftarrow}(1-kx/n)\right \}  }\right)  ^{4}\right]
\leq C(y-x),
\]
for any $0\leq x,y\leq1+\varepsilon$ and $n\in \mathbb{N}$.
\end{description}

For notation simplicity, we only present the asymptotic normality of
$\widehat{dev}_{\tau_{n}}(H)$ when $\gamma$ is estimated by the Hill
estimator. Similar results hold true for other estimators as long as the
functional $T$ satisfies certain conditions; see Corollaries 3.2 and 3.3 of
\cite{drees2000weighted}.

\begin{lemma}
\label{dep int}Let $\left(  X_{1},X_{2},...,X_{n}\right)  $ be a stationary $\beta$-mixing
time series satisfying the condition (C$_{R}$) with a continuous common
marginal distribution function $F$ such that its tail quantile function
$U\in2\mathrm{RV}_{\gamma,\rho}$ with $0<\gamma<1/2$, $\rho \leq0$ and
auxiliary function $A(t)$. Suppose that $k$ is a sequence such that as
$n\rightarrow \infty$, $k\rightarrow \infty$, $k/n\rightarrow \infty$, and
$\sqrt{k}A(n/k)\rightarrow \lambda \in \mathbb{[}0,\infty)$. We have as
$n\rightarrow \infty$
\begin{equation}
\sqrt{k}\left(  \widehat{\gamma}_{H}-\gamma,\frac
{\widehat{q}_{\tau_{n}}}{q_{\tau_{n}}}-1\right)  \overset{d}{\longrightarrow} \left(  N\left(  \frac{\lambda}{1-\rho
},\sigma_{H,\gamma}^{2}\right)  ,N(0,\gamma^{2}r(1,1))\right)  ,\label{di}%
\end{equation}
where $\sigma_{H,\gamma}^{2}=\gamma^{2}\left(  \int_{[0,1]^{2}}\left(  st\right)  ^{-1}r(t,s)dsdt-2\int
_{0}^{1}t^{-1}r(t,1)dt+r(1,1)\right) $ and $r(\cdot,\cdot)$ is
the covariance function defined in the condition (C$_{R}$). The limiting
distribution $\left(  N\left(  \frac{\lambda}{1-\rho},\sigma_{H,\gamma}%
^{2}\right)  ,N(0,\gamma^{2}r(1,1))\right)  $ is jointly normal with
covariance $\varsigma=\gamma^{2}\left(  \int_{0}^{1}t^{-1}%
r(t,1)dt-r(1,1)\right) $. 
\end{lemma}

\bigskip

Lemma \ref{dep int} verifies the joint convergence in (\ref{limit}) of Theorem \ref{int},
then we can immediately obtain the asymptotic normality of $\widehat
{dev}_{\tau_{n}}(H)$ for $\beta$-mixing series.

\begin{theorem}
\label{AN int}Under the conditions of Lemma \ref{dep int}, if $\sqrt{n\left(
1-\tau_{n}\right)  }q_{\tau_{n}}^{-1}\rightarrow \lambda_{1}\in \mathbb{R}$ and
$\sqrt{n\left(  1-\tau_{n}\right)  }A\left(  \left(  1-\tau_{n}\right)
^{-1}\right)  \rightarrow \lambda_{2}\in \mathbb{R}$, then we have that as
$n\rightarrow \infty$,
\[
\sqrt{n\left(  1-\tau_{n}\right)  }\left(  \frac{\widehat{dev}_{\tau_{n}}(H)%
}{dev_{\tau_{n}}}-1\right)  \overset{d}{\longrightarrow}N(\mu,\sigma^{2}),
\]
where $\mu=m(\gamma)\frac{\lambda_{2}}{1-\rho}-\left(  \lambda_{1}\xi
_{1}+\lambda_{2}\xi_{2}\right)  $ and $\sigma^{2}=m^{2}(\gamma)\sigma
_{H,\gamma}^{2}+2m(\gamma)\varsigma+\gamma^{2}r(1,1)$ with $\varsigma$ defined in Lemma
\ref{dep int}.
\end{theorem}

The asymptotic normality of the extreme level estimator $\widehat{dev}_{p_{n}%
}^{\ast}(H)$ follows from Theorems \ref{extreme} and \ref{AN int}:

\begin{theorem}
Under the conditions of Theorem \ref{AN int}, as $n\rightarrow \infty$, let
$\tau_{n},p_{n}\rightarrow1$ with $n\left(  1-\tau_{n}\right)  \rightarrow
\infty$ and $n\left(  1-p_{n}\right)  \rightarrow c<\infty$, and further
assume $\sqrt{n\left(  1-\tau_{n}\right)  }/\log \left[  \left(  1-\tau
_{n}\right)  /\left(  1-p_{n}\right)  \right]  \rightarrow \infty$. We have as
$n\rightarrow \infty$,
\[
\frac{\sqrt{n\left(  1-\tau_{n}\right)  }}{\log \left[  \left(  1-\tau
_{n}\right)  /\left(  1-p_{n}\right)  \right]  }\left(  \frac{\widehat
{dev}_{p_{n}}^{\ast}(H)}{dev_{p_{n}}}-1\right)  \overset{d}{\longrightarrow
}N\left(  \frac{\lambda_{2}}{1-\rho},\sigma_{H,\gamma}^{2}\right)  .
\]
\end{theorem}
}

\section{Simulation}\label{simulation}

In this section, we examine the performance of the proposed estimators through
simulations. The intermediate level estimator for the deviatile with
the Hill estimator for $\gamma$ is denoted by $\widehat{dev}_{\tau_{n}}(H)$,
and the extreme-level estimator with the Hill estimator is denoted by $\widehat
{dev}_{p_{n}}^{\ast}(H)$. The underlying risk is assumed to follow either a
Pareto($\alpha,\theta$) or Student's $t_\alpha$-distribution.
\subsection{Independence}
We first investigate the performance of both estimators when the data are i.i.d. observations.
\subsubsection{Intermediate level}

We first consider the
Pareto($\alpha,\theta$) distribution. In the simulation
example, we choose the level $\tau_{n}=1-k/n$ to estimate the deviatile, which is a common choice for the
intermediate level in the literature. {As discussed in \cite{ferreira2003optimising} and \cite{daouia2018estimation}, the intermediate sequence $k(n)$ can be chosen from
$\log n$ to $n/\log n$.}
%\begin{comment}
%Note that $A\left(1 /\left(1-\tau_{n}\right)\right) \sim \left(1-\tau_{n}\right)^{\frac{1}{\alpha}}$. The condition $\sqrt{n\left(1-\tau_{n}\right)} A\left(\left(1-\tau_{n}\right)^{-1}\right) \rightarrow \lambda \in \mathbb{R}$ as
%$n \rightarrow \infty$ in this paper reduces to
%$$
%\sqrt{n}\left(1-\tau_{n}\right)^{\frac{1}{\alpha}+\frac{1}{2}} \rightarrow \lambda \in \mathbb{R}, \quad n \rightarrow \infty
%$$
%Note that $\overline{F}\left(1-\tau_{n}\right)=\left(1-\tau_{n}\right)^{-\alpha}$ and $q_{\tau_{n}}=\left(1-\tau_{n}\right)^{-1 / \alpha}-1 $. Here$\sqrt{n\left(1-\tau_{n}\right)} / q_{\tau_n}\rightarrow$
%$\lambda$  reduces to
%$$
%\sqrt{n}\left(1-\tau_{n}\right)^{\frac{1}{\alpha}+\frac{1}{2}}\rightarrow \lambda  \text { as } n  \rightarrow \infty
%$$
%Hence it is possible to find $\tau_{n}$ to satisfy the two conditions  above simultaneously.
%In addition, since simulations are performed under finite sample, we choose the values of $\tau_{n}$ such
%that $\tau_{n}=1-k/n$, which satisfies the condition  above. This is a common choice for $\tau_{n}$ and
%for example it was adopted in \cite{daouia2018estimation}.
%\end{comment}

We draw 1000 random samples with sample size
$n=1000$, $4000$ and $10000$ to investigate the effect of sample
size. We also
compare the estimations for the Pareto($\alpha,\theta$) distribution with different heavy
tailedness, i.e., $\alpha=3$ and $5$ and $\theta=1$. Here, we follow the common practice in
the extreme value statistics to choose $k$ from the first stable region of the
Hill plot.

%For example, when example size is 1000, the estimator of quantile is not stable, which causes that the results will vary in square magnitude. And for $\hat{\gamma}$, especially when we set it close to $\frac{1}{2}$. The results of estimator $\hat{\gamma}$ may be larger than $\frac{1}{2}$. Then the adjusted variantile will be negative. Here we consider to give some limit to the $\hat{\gamma}$. We only take the $\hat{\gamma}$ which is smaller than $\frac{1}{2}$ and take the average value. For quantile, we also take the average.

In Table \ref{table1}, we report the true values of $dev_{\tau_{n}}$ {(in which the true value of the expectile is calculated numerically with R)}, and the
sample means, sample standard deviations and mean squared errors (MSEs) of $\frac{\widehat{dev}_{\tau_{n}}(H)}{dev_{\tau_{n}}}$. Overall, the estimator $\widehat{dev}%
_{\tau_{n}}(H)$ performs well. It performs
better for larger sample size and the lighter-tailed distribution.

\begin{table}[ptb]
\centering
%\onehalfspacing	
\setlength{\tabcolsep}{3mm}{\
\begin{tabular}
[c]{ccccc}\hline
		$(n,\alpha,k)$ & $\tau_{n}$ 
	  &${dev}_{\tau_n}$&$\frac{\widehat{dev}_{\tau_n}(H)}{dev_{\tau_n}}$ & MSE \\\hline
$(1000,3,50)$ & $0.95$&$2.9759$ & $2.1395 (3.2003)$ & $11.5262$ \\ 
$(1000,3,30)$ & $0.97$&$3.6631$ & $1.7018(1.5947)$ & $3.0323$ \\ 
$(1000,3,10)$ & $0.99$ &$5.6010$ & $1.3314(1.8360)$ & $3.4766$ \\ 
$(4000,3,200)$ & $0.95$ &$2.9759$ & $2.3923(1.9202)$ & $5.6217$ \\ 
$(4000,3,120)$ & $0.97$ &$3.6631$ & $1.7807(1.0490)$ & $1.7087$ \\ 
$(4000,3,40)$ & $0.99$ &$5.6010$ & $1.4125(0.9987)$ & $1.1665$ \\ 
$(10000,3,500)$ & $0.95$&$2.9759$ & $2.2691(1.6456)$ & $4.3160$ \\ 
$(10000,3,300)$ & $0.97$ & $3.6631$ & $1.6748(1.2987)$ & $2.1404$ \\ 
$(10000,3,100)$ & $0.99$ & $5.6010$ & $1.3258(0.7755)$ & $0.7069$ \\ 
$(1000,5,50)$ & $0.95$ &$0.9562$ & $1.5664(0.7861)$ & $0.9382$ \\ 
$(1000,5,30)$ & $0.97$ & $1.1345$ & $1.4351(0.6487)$ & $0.6096$ \\ 
$(1000,5,10)$ & $0.99$ &$1.5930$ & $1.3325(1.1040)$ & $1.3282$ \\ 
$(4000,5,200)$ & $0.95$ &$0.9562$ & $1.4419(0.1797)$ & $0.2275$ \\ 
$(4000,5,120)$ & $0.97$ & $1.1345$ & $1.3167(0.1696)$ & $0.1291$ \\ 
$(4000,5,40)$ & $0.99$ &$1.5930$ & $1.2007(0.2061)$ & $0.0827$ \\ 
$(10000,5,500)$ & $0.95$ &$0.9562$ & $1.4272 (0.1142)$ & $0.1955$ \\ 
$(10000,5,300)$ & $0.97$ &$1.1345$ & $1.2961(0.0967)$ & $0.0970$ \\ 
$(10000,5,100)$ & $0.99$ &$1.5930$ & $1.1734(0.1048)$ & $0.0410$ \\ \hline
\end{tabular}
} \caption{{\protect \small Based on 1000 samples of size $n$ of the
Pareto($\alpha,\theta$) distribution, the true values of $dev_{\tau_{n}}$, and the
sample means, sample standard deviations and MSEs of  $\frac{\widehat{dev}_{\tau_{n}}(H)}{dev_{\tau_{n}}}$
are presented for various thresholds $k$. }}%
\label{table1}%
\end{table}

Next, we consider	the Student's $t_{\alpha}$-distribution with the density function defined in \eqref{st}. Similar to the simulation example of the Pareto($\alpha,\theta$) distribution, we choose
$\tau_{n}=1-k/n$ and draw 1000 random samples with sample size $n=1000$,
$4000$ and $10000$ from the Student's $t_{\alpha}$-distribution for
$\alpha=3$ and $5$. In Table \ref{table2}, we report the true values of
$dev_{\tau_{n}}$ {(in which the true value of the expectile is calculated numerically with R)}, and the sample means, sample standard deviations and MSEs of $\frac{\widehat{dev}_{\tau_{n}}(H)}{dev_{\tau_{n}}}$. Overall, the estimator $\widehat{dev}%
_{\tau_{n}}(H)$ performs well for the Student's $t_{\alpha}$-distribution,
even for smaller sample sizes. It performs better for the lighter-tailed case.
	
\begin{table}[ptb]
\centering
%\onehalfspacing	
\setlength{\tabcolsep}{3mm}{\
\begin{tabular}
[c]{ccccc}\hline
		$(n,\alpha,k)$ & $\tau_{n}$ 
	  &${dev}_{\tau_n}$&$\frac{\widehat{dev}_{\tau_n}(H)}{dev_{\tau_n}}$ & MSE\\
	  \hline
$(1000,3,50)$ & $0.95$ & $3.9685$ & $1.4301(1.2003)$ & $1.6242$ \\ 
$(1000,3,30)$ & $0.97$ & $4.6813$ & $1.3407(1.4762)$ & $2.2930$ \\ 
$(1000,3,10)$ & $0.99$ & $6.6864$ & $1.2450(1.3797)$ & $1.9616$ \\ 
$(4000,3,200)$ & $0.95$ & $3.9685$ & $1.2339(0.2912)$ & $0.1394$ \\ 
$(4000,3,120)$ & $0.97$ & $4.6813$ & $1.1314(0.2670)$ & $0.0885$ \\ 
$(4000,3,40)$ & $0.99$ & $6.6864$ & $1.1339(0.6583)$ & $0.4508$ \\ 
$(10000,3,500)$ & $0.95$ & $3.9685$ & $1.1989(0.1424)$ & $0.0598$ \\ 
$(10000,3,300)$ & $0.97$ & $4.6813$ & $1.0976(0.1301)$ & $0.0264$ \\ 
$(10000,3,100)$ & $0.99$ & $6.6864$ & $1.0582(0.2318)$ & $0.0571$ \\ 
$(1000,5,50)$ & $0.95$ & $2.5862$ & $1.0381(0.1610)$ & $0.0273$ \\ 
$(1000,5,30)$ & $0.97$ & $2.9097$ & $1.0104(0.2032)$ & $0.0413$ \\ 
$(1000,5,10)$ & $0.99$ & $3.7075$ & $1.0384(0.3597)$ & $0.1307$ \\ 
$(4000,5,200)$ & $0.95$ & $2.5862$ & $1.0167(0.0744)$ & $0.0058$ \\ 
$(4000,5,120)$ & $0.97$ & $2.9097$ & $0.9880(0.0718)$ & $0.0053$ \\ 
$(4000,5,40)$ & $0.99$ & $3.7075$ & $0.9846(0.0988)$ & $0.0100$ \\ 
$(10000,5,500)$ & $0.95$ & $2.5862$ & $1.0098(0.0426)$ & $0.0019$ \\ 
$(10000,5,300)$ & $0.97$ & $2.9097$ & $0.9811(0.0436)$ & $0.0023$ \\ 
$(10000,5,100)$ & $0.99$ & $3.7075$ & $0.9751(0.0566)$ & $0.0038$ \\ 
		\hline
\end{tabular}
} \caption{{\protect \small Based on 1000 samples of size $n$ of
the Student's $t_\alpha$-distribution, the true values of $dev_{\tau_{n}}$, and the
sample means, sample standard deviations and MSEs of  $\frac{\widehat{dev}_{\tau_{n}}(H)}{dev_{\tau_{n}}}$
are presented for various thresholds $k$. }}%
\label{table2}%
\end{table}

\subsubsection{Extreme level}

In the simulation study of the estimation at the extreme level, we again consider
1000 random samples with sample size $n=1000$, $4000$ and $10000$ generated
from the Pareto($\alpha,1$) and Student's $t_{\alpha}$-distributions with
$\alpha=3$ and $5$. The extreme level $p_{n}$ is set at $0.9996$. Let
$\tau_{n}=1-k/n$, which satisfies the conditions of Theorem \ref{extreme}.

The values of $\widehat{dev}_{p_{n}}^{\ast
}(H)$, the sample means, sample standard
deviations and MSEs of  $\frac{\widehat{dev}_{p_{n}}^{\ast
}(H)}{dev_{p_{n}}}$ for the Pareto($\alpha,1$) distribution are reported in Table
\ref{table3}. For comparison, the true value of $dev_{0.9996}$  for
the Pareto($\alpha,1$) distribution is $17.8283$ when $\alpha=3$ and
is $3.7609$ when $\alpha=5$. We repeat similar experiments for
the Student's $t_{\alpha}$-distribution in Table
\ref{table4}. The true value of $dev_{0.9996}$  for the Student's $t_{\alpha}$-distribution is $19.3173$ when $\alpha=3$ and is $7.2585$ when $\alpha=5$. {Overall, the estimator $\widehat{dev}_{p_{n}}^{\ast}(H)$ performs better for light-tailed risks and better for the Student's $t_{\alpha}$-distribution. The selection of $k$ affects the performance of the estimator a lot and the investigation for an optimal choice of $k$ is left for future research.}  

\begin{table}[ptb]
\centering
%\onehalfspacing	
\setlength{\tabcolsep}{3mm}{\
\begin{tabular}
[c]{ccccc}\hline
			$(n,\alpha,k)$ & $\tau_{n}$ & $\widehat{dev}_{p_{n}}^{\ast
}(H)$ & $\frac{\widehat{dev}_{p_{n}}^{\ast}(H)}{dev_{p_{n}}}$&MSE   \\
			\hline

$(1000,3,50)$ & $0.95$ & $55.5700$ & $3.1169(5.0227)$ & $29.6736$ \\ 
$(1000,3,30)$ & $0.97$ & $47.6662$ & $2.6736(5.0535)$ & $28.3064$ \\ 
$(1000,3,10)$ & $0.99$ & $27.8433$ & $1.5617(1.9913)$ & $4.2758$ \\ 
$(4000,3,200)$ & $0.95$ & $68.9901$ & $3.8697(4.2013)$ & $25.8658$ \\ 
$(4000,3,120)$ & $0.97$ & $47.4150$ & $2.6595(2.6586)$ & $9.8148$ \\ 
$(4000,3,40)$ & $0.99$ & $30.5550$ & $1.7138(1.7233)$ & $3.4762$ \\ 
$(10000,3,500)$ & $0.95$ & $68.1895$ & $3.8248(2.6942)$ & $15.2304$ \\ 
$(10000,3,300)$ & $0.97$ & $43.3922$ & $2.4339(2.6096)$ & $8.8593$ \\ 
$(10000,3,100)$ & $0.99$ & $29.1761$ & $1.6365(1.4077)$ & $2.3848$ \\ 
$(1000,5,50)$ & $0.95$ & $10.4935$ & $2.7901(2.8125)$ & $11.1067$ \\ 
$(1000,5,30)$ & $0.97$ & $8.1097$ & $2.1563(1.8325)$ & $4.6919$ \\ 
$(1000,5,10)$ & $0.99$ & $6.5000$ & $1.7283(2.4544)$ & $6.5484$ \\ 
$(4000,5,200)$ & $0.95$ & $8.6064$ & $2.2884(0.5868)$ & $2.0038$ \\ 
$(4000,5,120)$ & $0.97$ & $6.7899$ & $1.8054(0.4543)$ & $0.8548$ \\ 
$(4000,5,40)$ & $0.99$ & $5.2311$ & $1.3909(0.4795)$ & $0.3825$ \\ 
$(10000,5,500)$ & $0.95$ & $8.4539$ & $2.2478(0.3367)$ & $1.6703$ \\ 
$(10000,5,300)$ & $0.97$ & $6.6487$ & $1.7678(0.2740)$ & $0.6646$ \\ 
$(10000,5,100)$ & $0.99$ & $5.0348$ & $1.3387(0.2510)$ & $0.1776$ \\ 
			\hline
			
\end{tabular}
} \caption{{\protect \small Based on 1000 samples of size $n$ of
the Pareto($\alpha,\theta$) distribution, the values of $\widehat{dev}_{p_n}^{\ast}(H)$, and the
sample means, sample standard deviations and MSEs of  $\frac{\widehat{dev}_{p_{n}}^{\ast}(H)}{dev_{p_{n}}}$ with $p_n=0.9996$
are presented for various thresholds $k$. }}%
\label{table3}%
\end{table}
	
\begin{table}[ptb]
\centering
%\onehalfspacing	
\setlength{\tabcolsep}{3mm}{\
\begin{tabular}
[c]{ccccc}\hline
			$(n,\alpha,k)$ & $\tau_{n}$ & $\widehat{dev}_{p_{n}}^{\ast}(H)$ & $\frac{\widehat{dev}_{p_{n}}^{\ast}(H)}{dev_{p_{n}}}$&MSE   \\
			\hline
$(1000,3,50)$ & $0.95$ & $44.2138$ & $2.2888(2.4963)$ & $7.8858$ \\ 
$(1000,3,30)$ & $0.97$ & $38.0555$ & $1.9700(3.4765)$ & $13.0143$ \\ 
$(1000,3,10)$ & $0.99$ & $27.4548$ & $1.4213(1.7322)$ & $3.1745$ \\ 
$(4000,3,200)$ & $0.95$ & $36.4550$ & $1.8872(0.7870)$ & $1.4058$ \\ 
$(4000,3,120)$ & $0.97$ & $28.8701$ & $1.4945(0.5949)$ & $0.5981$ \\ 
$(4000,3,40)$ & $0.99$ & $25.5117$ & $1.3207(0.9225)$ & $0.9530$ \\ 
$(10000,3,500)$ & $0.95$ & $34.5258$ & $1.7873(0.3521$ & $0.7437$ \\ 
$(10000,3,300)$ & $0.97$ & $27.3741$ & $1.4171(0.3040)$ & $0.2663$ \\ 
$(10000,3,100)$ & $0.99$ & $22.6406$ & $1.1720(0.3505)$ & $0.1523$ \\ 
$(1000,5,50)$ & $0.95$ & $13.1665$ & $1.8139(0.8570)$ & $1.3962$ \\ 
$(1000,5,30)$ & $0.97$ & $10.8009$ & $1.4880(0.6133)$ & $0.6140$ \\ 
$(1000,5,10)$ & $0.99$ & $9.3795$ & $1.2922(0.8419)$ & $0.7935$ \\ 
$(4000,5,200)$ & $0.95$ & $12.2619$ & $1.6893(0.2835)$ & $0.5554$ \\ 
$(4000,5,120)$ & $0.97$ & $10.0120$ & $1.3793(0.2445)$ & $0.2036$ \\ 
$(4000,5,40)$ & $0.99$ & $8.2725$ & $1.1397(0.2521)$ & $0.0830$ \\ 
$(10000,5,500)$ & $0.95$ & $12.1204$ & $1.6698(0.1779)$ & $0.4803$ \\ 
$(10000,5,300)$ & $0.97$ & $9.8982$ & $1.3637(0.1509)$ & $0.1550$ \\ 
$(10000,5,100)$ & $0.99$ & $8.1310$ & $1.1202(0.1483)$ & $0.0364$ \\ 
			\hline
			
\end{tabular}
} \caption{{\protect \small Based on 1000 samples of size $n$ of
the Student's $t_\alpha$-distribution, the values of $\widehat{dev}_{p_n}^{\ast}(H)$, and the
sample means, sample standard deviations and MSEs of  $\frac{\widehat{dev}_{p_{n}}^{\ast}(H)}{dev_{p_{n}}}$ with $p_n=0.9996$
are presented for various thresholds $k$. }}%
\label{table4}%
\end{table}

{
\subsection{Serial dependence}
In this subsection we examine the performance of the estimators when the
sample is a $\beta$-mixing process. Here we consider the GARCH(1,1) process
defined as
\begin{equation}
\left \{
\begin{array}
[c]{l}%
X_{t}=\sigma_{t}\varepsilon_{t},\\
\sigma_{t}^{2}=\alpha_{0}+\alpha_{1}X_{t-1}^{2}+\beta_{0}\sigma_{t-1}^{2},
\end{array}
\right.  \label{garch}%
\end{equation}
where $\varepsilon_{t}$ are i.i.d. innovations with zero mean and unit
variance. The stationary GARCH(1,1) series satisfies the $\beta$-mixing
condition and the regularity conditions (C$_{R}$); see for example
\cite{drees2000weighted}. Following the modeling for the financial data used
in \cite{de2016adapting}
and \cite{daouia2018estimation}, we assume the innovations $\varepsilon_{t}$ being the
standardized Student's $t$-distribution. The parameters are set as follows:
the degree of freedom of the Student's $t$-distribution is $6.54$ and
$\alpha_{0}=0.0181$, $\alpha_{1}=0.1476$, $\beta_{0}=0.8497$. This is to
be consistent with those estimated with the real data in Section \ref{real}. The true values
of the deviatile are computed by the Monte Carlo method based on $1000$
samples of GARCH(1,1) with size of $10^{6}$ using the following
estimator
\[
\widetilde{dev}_{\tau_{n}}=\left(  \frac{1}{n}\sum_{i=1}^{n}\frac{\tau_{n}%
}{1-\tau_{n}}\left(  X_{i}-\widetilde{e}_{\tau_{n}}\right)  _{+}^{2}+\left(
X_{i}-\widetilde{e}_{\tau_{n}}\right)  _{-}^{2}\right)  ^{1/2},
\]
where $\widetilde{e}_{\tau_{n}}$ is obtained by using the function \texttt{expectile} from the R package \texttt{expectreg}. 
Both the intermediate and extreme level estimators $\widehat{dev}_{\tau_{n}%
}(H)$ and $\widehat{dev}_{p_{n}}^{\ast}(H)$ are computed based on 1000 random
samples with sample size of $n=1000$ and $5000$ generated from the GARCH(1,1)
process. The sample size of $5000$ is comparable with that of the real data
study in Section \ref{real}. Similar as before, $\tau_{n}=1-k/n$. 

In Table \ref{tablenewadd_1}, we report the true values of $dev_{\tau_{n}}$,
and the sample means, sample standard deviations and MSEs of $\frac
{\widehat{dev}_{\tau_{n}}(H)}{dev_{\tau_{n}}}$. We can see that the
performance is greatly influenced by the selection of $k$. It is expected that
the MSEs are larger than those for the independent data in Table \ref{table2}
as the serial dependence in data increases the variance of the limiting
distribution.

\begin{table}[ptb]
	\centering
	%\onehalfspacing	
	\setlength{\tabcolsep}{3mm}{\
		\begin{tabular}
			[c]{ccccc}\hline
			$(n,k)$ & $\tau_{n}$ &${dev}_{\tau_n}$&$\frac{\widehat{dev}_{\tau_n}(H)}{dev_{\tau_n}}$ & MSE\\
			\hline
			$(1000,50)$ & $0.95$ & $5.0938$ & $0.9221(1.4048)$ & $1.9771$ \\ 
			$(1000,30)$ & $0.97$ & $5.8360$ & $0.9451(1.3497)$ & $1.8226$ \\ 
			$(1000,20)$ & $0.98$ & $6.8071$ & $0.9230(1.5435)$ & $2.3858$ \\ 
			$(5000,150)$ & $0.97$ & $5.8360$ & $1.1301(0.9934)$ & $1.0026$ \\ 
			$(5000,100)$ & $0.98$ & $6.8071$ & $1.0343(1.1510)$ & $1.3245$ \\ 
			$(5000,50)$ & $0.99$ & $9.4925$ & $0.9973(1.6800)$ & $2.8192$ \\ 
		   \hline
			
		\end{tabular}
	} \caption{{\protect \small Based on 1000 samples of size $n$ of
			the GARCH(1,1),the true values of $dev_{\tau_{n}}$, and the
			sample means, sample standard deviations and MSEs of  $\frac{\widehat{dev}_{\tau_{n}}(H)}{dev_{\tau_{n}}}$
			are presented for various thresholds $k$.  }}%
	\label{tablenewadd_1}%
\end{table}

For the extreme level estimator, we set the extreme level $p_{n}$ at $0.9996$.
The true value of $dev_{0.9996}$ is 27.7818. The values of $\widehat
{dev}_{p_{n}}^{\ast}(H)$, sample means, sample standard deviations and MSEs of
$\frac{\widehat{dev}_{\tau_{n}}(H)}{dev_{\tau_{n}}}$ are reported in Table
\ref{tablenewadd_2}. We can see that the performance is greatly influenced by
the selection of $k$. As the variance is increased for the limiting
distribution in the dependent data case, it is of great importance to develop
a scheme to optimal select $k$ so that the bias and the variance are balanced
for the estimation of the risk measure at the extreme level, which is left for
the future research.
\begin{table}[ptb]
	\centering
	%\onehalfspacing	
	\setlength{\tabcolsep}{3mm}{\
		\begin{tabular}
			[c]{ccccc}\hline
				$(n,k)$& $\tau_{n}$ & $\widehat{dev}_{p_{n}}^{\ast}(H)$ & $\frac{\widehat{dev}_{p_{n}}^{\ast}(H)}{dev_{p_{n}}}$&MSE   \\
				\hline
			$(1000,50)$ & $0.95$ & $38.6006$ & $1.3894(2.6061)$ & $6.9349$ \\ 
			$(1000,30)$ & $0.97$ & $32.9287$ & $1.1852(2.3072)$ & $5.3519$ \\ 
			$(1000,20)$ & $0.98$ & $21.3725$ & $0.7693(2.6373)$ & $7.0013$ \\ 
			$(5000,150)$ & $0.97$ & $53.5653$ & $1.9280(2.7564)$ & $8.4490$ \\ 
			$(5000,100)$ & $0.98$ & $45.2008$ & $1.6269(3.7811)$ & $14.6727$ \\ 
			$(5000,50)$ & $0.99$ & $30.4863$ & $1.0973(2.4425)$ & $5.9685$ \\ 
			\hline
			
		\end{tabular}
	} \caption{{\protect \small Based on 1000 samples of size $n$ of
			the GARCH(1,1), the values of $\widehat{dev}_{p_n}^{\ast}(H)$, and the
			sample means, sample standard deviations and MSEs of  $\frac{\widehat{dev}_{p_{n}}^{\ast}(H)}{dev_{p_{n}}}$ with $p_n=0.9996$
			are presented for various thresholds $k$. }}%
	\label{tablenewadd_2}%
\end{table}
}

\section{Real Data Applications} \label{real}

In this section, we apply the proposed intermediate and extreme level
estimators of the deviatile to the real stock data under both independent and serial dependence cases.  Consider the daily
log-losses (i.e., negative log-return) of S\&P 500 index from January 1, 2000,
to December 31, 2019. There are 5030 observations in total, 2514 observations for the period from 2000 to 2009 and 2516 for the period from 2010 to 2019.
%To reduce the serial
%correlation, we consider the observations on even days only and apply our
%method by regarding those observations as independent. The almost independence
%among every other day data is supported by various empirical studies on the
%tail index for the financial data; see e.g. \cite{mcneil1998calculating},
%\cite{poon2003modelling} and \cite{hamidieh2009estimation}. By considering the
%every other day data, the data set contains 1259 observations from November
%1st, 2009, to November 1st, 2019, and 1132 observations from January 1st,
%2000, to January 1st, 2009.

We show the daily log-losses of the S\&P 500 index in Figure \ref{label_for_cross_ref_ts2000} for the whole period (2000-2019) and focus
on the financial crisis period (2007-2009) in Figure \ref{label_for_cross_ref_ts2009}. One can see frequent
large losses during the financial crisis period. 

	\begin{figure}[ptb]
		\centering
		\subfigure[2000-2019
		]{
			\includegraphics[width=2.7in]{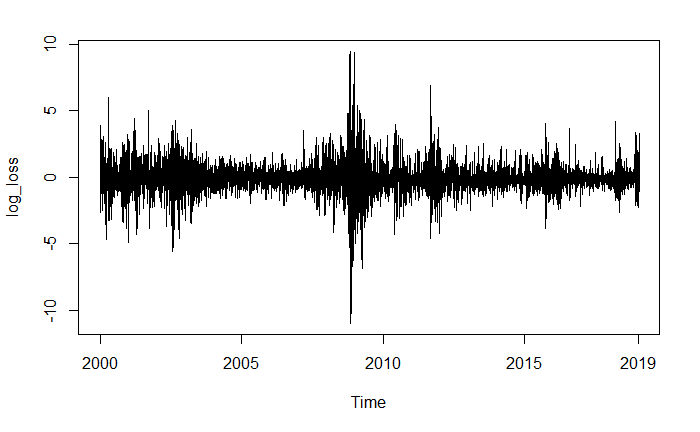}
			\label{label_for_cross_ref_ts2000}
		}
		\subfigure[ 2007-2009
		]{
			\includegraphics[width=2.7in]{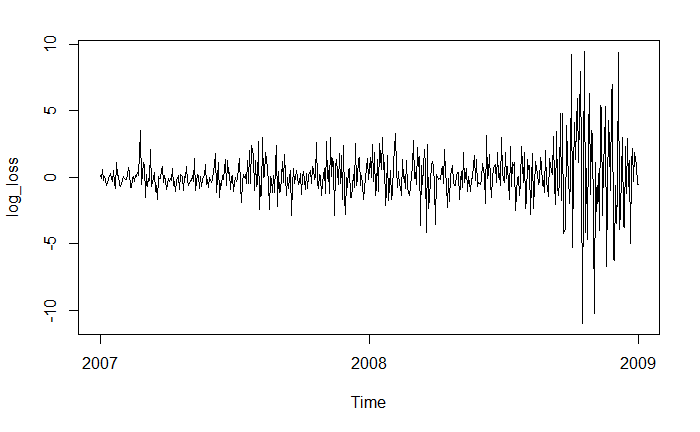}
			\label{label_for_cross_ref_ts2009}
		}
		\quad    %用 \quad 来换行
		\caption{Daily log-losses of the S$\&$P 500 index.}
		\label{fig.ts}
	\end{figure}

{First, we ignore the serial dependence exhibited in the daily log-losses but treat the data to be independent and identically distributed.} We compare the deviatile with VaR, ES and expectile at confidence levels of $0.95$, $0.99$, and $0.999$. When
calculating the risk measures at the level of $0.95$ and $0.99$, we apply the
intermediate-level estimators. That is, $\widehat{dev}_{\tau}(H)$ for the deviatile, $\widehat{q}_{\tau}=X_{n-\left \lfloor n\left(  1-\tau \right)
\right \rfloor ,n}$ for VaR,
\[
\widehat{ES}_{\tau}(H)=\frac{1}{1-\widehat{\gamma}_{H}}\widehat{q}_{\tau}%
\]
for ES, and
\[
\widehat{e}_{\tau}(H)=\left(  \widehat{\gamma}_{H}^{-1}-1\right)
^{-\widehat{\gamma}_{H}}\widehat{q}_{\tau}%
\]
for the expectile. When calculating the level at $p=0.999$, we apply the extreme-level estimators by using extrapolation and setting $\tau=0.95$. That is,
$\widehat{dev}_{p}^{\ast}(H)$ for the deviatile, $\widehat{q}_{p}%
^{\ast}=\left(  \frac{1-p}{1-\tau}\right)  ^{-\widehat{\gamma}}\widehat
{q}_{\tau}$ for VaR,
\[
\widehat{ES}_{\tau}^{\ast}(H)=\left(  \frac{1-p}{1-{\tau}}\right)
^{-\widehat{\gamma}_{H}}\widehat{ES}_{\tau}(H)
\]
for ES, and
\[
\widehat{e}_{\tau}^{\ast}(H)=\left(  \frac{1-p}{1-{\tau}}\right)
^{-\widehat{\gamma}_{H}}\widehat{e}_{\tau}(H)
\]
for the expectile. \\

In Table \ref{table5}, we compute the four risk measures for the period
from 2000 to  2009. The estimated tail index $\widehat{\gamma}_{H}$ for this period is $0.3885$. Table \ref{table6} reports the risk measures for
the period from 2010 to 2019, and the estimated tail index $\widehat{\gamma}_{H}$ for this period is $0.2391$.  All the risk measures show that  the
financial market was more stable in the period from 2010 to 2019
(Table \ref{table6}) than in that from 2000 to 2009 (Table
\ref{table5}). We also observe that the deviatile is larger than VaR
and the expectile but is smaller than ES. Both ES and the deviatile better reflect the substantial risks during the financial crisis. {This can be understood from the point of view that both ES and deviatile are deviation measures in the definition of risk quadrangle and thus they are more sensitive to capture large oscillations of data.}

\begin{table}[ptb]
\centering
%\onehalfspacing	
\setlength{\tabcolsep}{3mm}{\
\begin{tabular}
[c]{ccccc}\hline
			$\tau$ &$\widehat{q}_{\tau}$ & $\widehat{e}_{\tau}(H)$  &$\widehat{ES}_{\tau}(H)$ & $\widehat{dev}_{\tau}(H)$   \\
		
			\hline
			0.95& 2.1774 &1.8257 & 3.5611& 3.8673\\
	
			0.99&3.9189&3.2859&6.4093 &6.9603\\
		
			0.999&9.9563 &8.3482&16.2835& 17.6833\\
			\hline
			\end{tabular}
} \caption{{\protect \small The estimated VaR, expectile, ES and
    deviatile of the daily log-losses of the S\&P 500 index from 2000 to 2009.} }%
\label{table5}%
\end{table}

\begin{table}[ptb]
\centering
%\onehalfspacing	
\setlength{\tabcolsep}{3mm}{\
\begin{tabular}
[c]{ccccc}\hline
			$\tau$ &$\widehat{q}_{\tau}$ & $\widehat{e}_{\tau}(H)$  &$\widehat{ES}_{\tau}(H)$ & $\widehat{dev}_{\tau}(H)$   \\
		
			\hline
		0.95& 1.5536 &1.1780 & 2.0419& 1.6309\\
		
			0.99& 2.8461& 2.1580 & 3.7407& 2.9877\\
		
			0.999& 3.9597 &3.0023&  5.2043& 4.1567\\
			\hline
			\end{tabular}
} \caption{{\protect \small The estimated VaR, expectile, ES and
    deviatile of the daily log-losses of the S\&P 500 index from 2010 to 2019.}}%
\label{table6}%
\end{table}
	%In two tables, we mainly show the risk measure including VaR, Expectile, conditional VAR and ad. Firstly, we can know $q_{\tau}$ is lager than $e_{\tau}$, which means $\gamma<\frac{1}{2}$. This is the condition we hope to meet to calculate ad. Secondly, compared to other risk measure, ad seems to be like cvar. And from the form of ad ,it is also similar to cVaR.
		\begin{figure}[ptb]
		\centering
		\subfigure[$\widehat{dev}_{0.95}(H)$
		]{
			\includegraphics[width=2.7in]{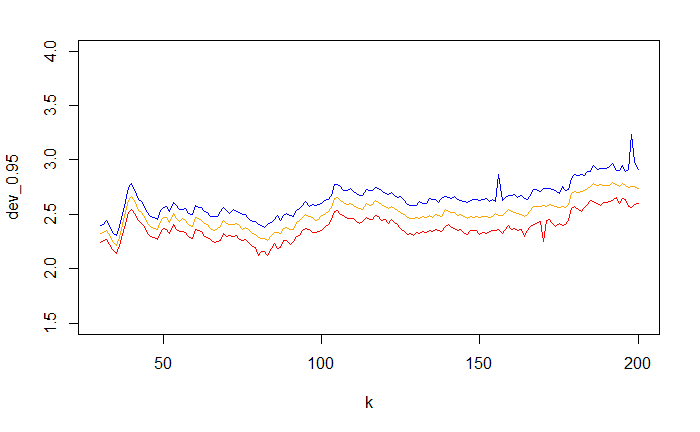}
			\label{bootstrap0.95}
		}
		\subfigure[ $\widehat{dev}^*_{0.999}(H)$
		]{
			\includegraphics[width=2.7in]{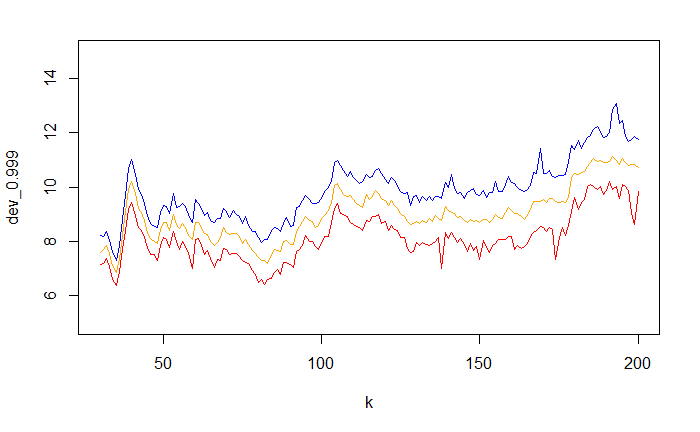}
			\label{bootstrap0.999}
		}
		\quad    %用 \quad 来换行
		\caption{$\widehat{dev}_{0.95}(H)$ and $\widehat{dev}^*_{0.999}(H)$ for the S\& P 500 index daily log-losses are plotted in yellow curves in the left and right panel, respectively. The lower and upper bounds of the corresponding 0.95 confidence interval are plotted in blue and red curves. }
		\label{fig.bootstrap}
	\end{figure}
{Chapter 3 of \cite{mcneil2015quantitative} summarized the stylized features of financial return series including heavy
tailed returns, volatility clustering and serial dependence. To model these
features of daily log-losses of S\&P 500 index, we following the same model
used in \cite{de2016adapting} and \cite{daouia2018estimation} that we fit the data with a GARCH(1,1) process defined in (\ref{garch})
with standardized Student's $t$-distribution. The estimated degree of freedom
of the Student's $t$-distribution is $\widehat{\nu}=6.54$, and $\widehat
{\alpha}_{0}=0.0181$, $\widehat{\alpha}_{1}=0.1476$, $\widehat{\beta}%
_{0}=0.8497$. We estimated the deviatile at $95\%$ level with intermediate
level estimator $\widehat{dev}_{0.95}(H)$ and estimate the $99.9\%$ level with
the extreme level estimator $\widehat{dev}_{0.999}^{\ast}(H)$ with various
choices of $k=30,31,...,200$ in the estimation of the Hill estimator. We also
construct the $0.95$ confidence interval for these estimators. Since there is
no explicit formula for calculating the asymptotic variance under the GARCH
model, we follow the similar procedure in \cite{de2016adapting} to estimate the asymptotic
variance with simulation. To be more specific, we use the block bootstrapping
method by applying \texttt{tsboot} in the package \texttt{boot} in R with the
block lengths chosen to have a geometric distribution (\texttt{sim} %
$=$ \texttt{geom}) with mean $200$. We obtain 100 bootstrapped estimates for
each estimator by repeating the bootstrapping procedure 100 times. The sample
variance is then calculated as an estimate of the asymptotic variance of the
underlying estimator. The estimates of deviatile together with the 0.95
confidence intervals are plotted against $k$ in Figure \ref{fig.bootstrap}.
Both estimators stay stable over different $k$. The intermediate estimator has
a narrower confidence interval overall. }
	
	\section{Conclusion} \label{conclusion}
	{In this paper, to better capture extreme risks, the adjusted
        standard-deviatile is defined by  making adjustments to the variantile.} By investigating the first- and
        second-order asymptotic expansions of the deviatile, we study
        efficient estimation methods for the deviatile at both
        intermediate and extreme levels.   The asymptotic normality is
        proved for both estimators when the sample is assumed to be i.i.d. and is exhibiting serial dependence (to be more specific, satisfying the $\beta$-mixing condition), respectively. Through simulation studies, we
        show that the proposed estimators are easy to implement and
        have good performance. Lastly, we compare the deviatile with
        other commonly used risk measures in the real financial market
        and find that the deviatile is capable of capturing extreme risks. 
	
\section*{Appendix}	
%\subsection*{Asymptotic expansions of the expectile}
{In the following proofs,   for a sequence of random variables $Y_{n}$, $n\in\N$,  the notation
$Y_{n}=o_{p}(1)$ means that for any $\varepsilon>0$, $\lim_{n\rightarrow \infty}%
\Pr(|Y_{n}|>\varepsilon)=0$.}

The asymptotic expansions of the deviatile rely on
those of the expectile, which are investigated in \cite{mao2015asymptotic} and
\cite{daouia2018estimation}. We present it in the following proposition.

\begin{proposition}
\label{p1}Assume that $U\in2\mathrm{RV}_{\gamma,\rho}$ with $0<\gamma<1$, $\rho \leq0$
and auxiliary function $A(t)$. Then as $\tau \uparrow1$, 
\[
\frac{\overline{F}\left(  e_{\tau}\right)  }{1-\tau}=\left(  \gamma
^{-1}-1\right)  (1+\eta(\tau)),\qquad \text{and}\qquad \frac{e_{\tau}}{q_{\tau}%
}=\left(  \gamma^{-1}-1\right)  ^{-\gamma}(1+r(\tau)),
\]
where%
\[
\eta(\tau)=-\frac{\left(  \gamma^{-1}-1\right)  ^{\gamma}}{q_{\tau}}%
({E}(X)+o(1))-\frac{\left(  \gamma^{-1}-1\right)  ^{-\rho}}{\gamma
(1-\rho-\gamma)}A\left(  (1-\tau)^{-1}\right)  (1+o(1)),
\]
and%
\[
r(\tau)=\frac{\gamma \left(  \gamma^{-1}-1\right)  ^{\gamma}}{q_{\tau}}%
({E}(X)+o(1))+\left(  \frac{\left(  \gamma^{-1}-1\right)  ^{-\rho}}%
{1-\rho-\gamma}+\frac{\left(  \gamma^{-1}-1\right)  ^{-\rho}-1}{\rho
}+o(1)\right)  A\left(  (1-\tau)^{-1}\right)  .
\]

\end{proposition}
 %Similar results on the expansions of $e_{\tau}$ can also be found in
 %\cite{mao2015risk}.
%Similar results on the expansions of $e_{\tau}$ can also be found in
%\cite{mao2015asymptotic} and \cite{mao2015risk}. 
%When $F$ belongs to the
%maximum domain of attraction of Gumbel, which is a more challenging case, the
%expansions of $e_{\tau}$ are provided in \cite{bellini2017risk}.

\bigskip
Now we are ready to prove the main results.
%	\subsection{Proof of Theorem \ref{1st}}
\begin{proof}[Proof of Theorem \ref{1st}]
Note that
$
E[(X-e_{\tau})^{2}]=E[((X-e_{\tau})_{-})^{2}]+E[((X-e_{\tau})_{+})^{2}].
$
We rewrite the deviatile as follows:
\begin{equation}
{dev}_{\tau}^{2}(X)=\frac{2\tau-1}{1-\tau}{E}\left[  (X-e_{\tau})_{+}%
^{2}\right]  +{E}\left[  (X-e_{\tau})^{2}\right]  .\label{rw}%
\end{equation}
First, we consider the expansion of ${E}\left[  (X-e_{\tau})_{+}^{2}\right]  $.
By integration by parts and a change of variable, we have
\[
{E}\left[  (X-e_{\tau})_{+}^{2}\right]  =2e^2_{\tau}\int_{1}^{\infty}%
\overline{F}(e_{\tau}x)(x-1)dx.
\]
Note that as $\tau \rightarrow1$, $e_{\tau}\rightarrow \infty$. By Potter's
bounds, (see, e.g., Proposition B.1.9(5) of \cite{de2007extreme}), for any
$\varepsilon,\delta>0$, there exists $t_{0}>0$ such that when $e_{\tau}>t_{0}$ and $x\geq1$, we have  
\[
\frac{\overline{F}(e_{\tau}x)}{\overline{F}(e_{\tau})}\leq(1+\varepsilon
)x^{-\frac{1}{\gamma}+\delta}.
\]
Then, by the dominated convergence theorem, for $0<\gamma<1/2$, we have
\begin{equation}
\lim_{\tau \rightarrow1}\frac{{E}\left[  (X-e_{\tau})_{+}^{2}\right]  }%
{e^2_{\tau}\overline{F}(e_{\tau})}=2\int_{1}^{\infty}x^{-\frac{1}{\gamma}%
}(x-1)dx=\frac{2\gamma^{2}}{\left(  1-2\gamma \right)  (1-\gamma)}.\nonumber
\end{equation}
By Proposition \ref{p1}, we can further expand ${E}\left[
(X-e_{\tau})_{+}^{2}\right]  $ as
\[
\lim_{\tau \rightarrow1}\frac{{E}(X-e_{\tau})_{+}^{2}}{(1-\tau)q_{\tau}^{2}%
}=(\gamma^{-1}-1)^{-2\gamma}\frac{2\gamma}{\left(  1-2\gamma \right)  }.
\]
Now, we turn to $E(X-e_{\tau})^{2}$. Note the rewriting
\begin{align*}
E\left[  (X-e_{\tau})^{2}\right]    & ={E}\left[  X^{2}\right]  -2e_{\tau}%
{E}\left[  X\right]  +e_{\tau}^{2}\\
& \sim e_{\tau}^{2}.
\end{align*}
By Proposition \ref{p1}, we have
\[
E[(X-e_{\tau})^{2}]\sim(\gamma^{-1}-1)^{-2\gamma}q_{\tau}^{2}.
\]
%	Combine the results, we get
%	\begin{align}
%	{avar}_{\tau}(X)&=(2\tau-1)q_{\tau}^2(\gamma^{-1}-1)^{1-2\gamma}(\frac{\gamma}{1-2\gamma}+\frac{\gamma}{\gamma-1})\notag\\
%	&+[{E}X^2-2\gamma({E}X)^2-2(\gamma^{-1}-1)^{-\gamma}q_{\tau}{E}X+(\gamma^{-1}-1)^{-2\gamma}q_{\tau}^2] \notag\\
%	&=(2\tau-1)q_{\tau}^2(\gamma^{-1}-1)^{-2\gamma}(\frac{1}{1-2\gamma})(1\notag\\
%	&+(1-2\gamma)(\frac{(\gamma^{-1}-1)^{2\gamma}}{q_{\tau}^2}{E}(X)^2-\frac{2\gamma(\gamma^{-1}-1)^{2\gamma}({E}X)^2}{q_{\tau}^2}-\frac{2(\gamma^{-1}-1)^{\gamma}{E}X}{q_{\tau}})  ), \text{as }\tau\uparrow 1.\notag
%	\end{align}
Combining the above results, we have
\begin{equation}
{dev}_{\tau}^{2}(X)\sim \frac{(\gamma^{-1}-1)^{-2\gamma}}{1-2\gamma}q_{\tau
}^{2}.\label{1stavar}%
\end{equation}
Now, by taking the square root on both sides of (\ref{1stavar}), the
first-order expansions of ${dev}_{\tau}(X)$ follow. 
\end{proof}	

\bigskip
\begin{comment}
\begin{proof}
[Proof of relation (\ref{rank})]We need to show when $\beta_{\gamma}>1$, which is
equivalent to show
\[
f(\gamma):=(\gamma^{-1}-1)^{-2\gamma}-(1-2\gamma)>0.
\]
The derivative of $f$ is
\[
f^{\prime}(\gamma)=(\gamma^{-1}-1)^{-2\gamma}\left(  -2\log(\gamma
^{-1}-1)+\frac{2}{1-\gamma}\right)  +2.
\]
Note that $f^{\prime}(\gamma)$ is an increasing function of $\gamma$ and
$f^{\prime}(0+)\rightarrow-\infty$ and $f^{\prime}(1/2)>0$. Thus $f(\gamma)$
is first decreasing at $0$ and then increasing. Since $f(0)=0$, there exists
one unique solution $\gamma^{\ast}$ in $(0,1/2)$ such that $f(\gamma)=0$. The
desired result then follows.
\end{proof}
\end{comment}
 \bigskip

\begin{proof}[Proof of Theorem \ref{2ndexp}]
By the same rewriting of the deviatile in \eqref{rw}, we first analyze
$E[(X-e_{\tau})_{+}^{2}]$. %%EDITOR'S NOTE: Please ensure that the
                           %%intended meaning has been maintained in
                           %%this edit.
Note
that
\begin{align}
E[(X-e_{\tau})_{+}^{2}] &  =2e_{\tau}^{2}\int_{1}^{\infty}\overline{F}%
(e_{\tau}x)(x-1)dx\nonumber \\
&  =2\overline{F}(e_{\tau})e_{\tau}^{2}\int_{1}^{\infty}\left(  x^{-\frac
{1}{\gamma}}+\frac{\overline{F}(e_{\tau}x)}{\overline{F}(e_{\tau})}%
-x^{-\frac{1}{\gamma}}\right)  (x-1)dx,\nonumber \\
&  =2\overline{F}(e_{\tau})e_{\tau}^{2}\left(  \frac{\gamma^{2}}%
{(1-2\gamma)(1-\gamma)}+\int_{1}^{\infty}\left(  \frac{\overline{F}(xe_{\tau
})}{\overline{F}(e_{\tau})}-x^{-\frac{1}{\gamma}}\right)  (x-1)dx\right)
\nonumber \\
&  =2c_{1}\overline{F}(e_{\tau})e_{\tau}^{2}\left(  1+\frac{1}{c_{1}}\int
_{1}^{\infty}\left(  \frac{\overline{F}(xe_{\tau})}{\overline{F}(e_{\tau}%
)}-x^{-\frac{1}{\gamma}}\right)  (x-1)dx\right)  \label{2nd}%
\end{align}
where $c_{1}=\frac{\gamma^{2}}{(1-2\gamma)(1-\gamma)}$. Now, we focus on the
integral term in (\ref{2nd}). Since $U\in2\mathrm{RV}_{\gamma,\rho}$ with
auxiliary function $A(\cdot)$, by Theorem 2.3.9 of \cite{de2007extreme}, for
any $\varepsilon,\delta>0$, there exists $t_{0}>0$ such that for $e_{\tau
}>t_{0}$ and $x\geq1$, we have
\[
\left \vert \frac{\frac{\overline{F}(xe_{\tau})}{\overline{F}(e_{\tau}%
)}-x^{-\frac{1}{\gamma}}}{A\left(  \frac{1}{\overline{F}(e_{\tau})}\right)
}-x^{-1/\gamma}\frac{x^{\rho/\gamma}-1}{\gamma \rho}\right \vert \leq \varepsilon
x^{-\frac{1}{\gamma}+\frac{\rho}{\gamma}+\delta}.
\]
Then, for $-\frac{1}{\gamma}+\frac{\rho}{\gamma}+2<0$, or equivalently
$0<\gamma<\frac{1-\rho}{2}$ which is ensured by $0<\gamma<\frac{1}{2}$, by the dominated convergence theorem, we have
\begin{align*}
\lim_{\tau \rightarrow1}\int_{1}^{\infty}\frac{\frac{\overline{F}(xe_{\tau}%
)}{\overline{F}(e_{\tau})}-x^{-\frac{1}{\gamma}}}{A\left(  \frac{1}%
{\overline{F}(e_{\tau})}\right)  }(x-1)dx &  =\int_{1}^{\infty}x^{-\frac
{1}{\gamma}}\frac{x^{\frac{\rho}{\gamma}}-1}{\gamma \rho}(x-1)dx\\
&  =\frac{(2-\rho)\gamma-3\gamma^{2}}{\left(  1-\rho-2\gamma \right)  \left(
1-\rho-\gamma \right)  \left(  1-2\gamma \right)  \left(  1-\gamma \right)  }.
\end{align*}
By Proposition \ref{p1} and $|A|\in \mathrm{RV}_{\rho}$, we have $A\left(
\frac{1}{\overline{F}(e_{\tau})}\right)  \sim(\gamma^{-1}-1)^{-\rho}A\left(
\frac{1}{1-\tau}\right)  $. Thus, as $\tau \rightarrow1$,
\[
E[(X-e_{\tau})_{+}^{2}]=2c_{1}\overline{F}(e_{\tau})e_{\tau}^{2}\left(
1+c_{2}A\left(  \frac{1}{1-\tau}\right)  (1+o(1))\right)  ,
\]
where $c_{2}=\frac{(2-\rho-3\gamma)\gamma^{1+\rho}}{\left(  1-\rho
-2\gamma \right)  \left(  1-\rho-\gamma \right)  \left(  1-2\gamma \right)
\left(  1-\gamma \right)  ^{1+\rho}c_{1}}$. Furthermore, by Proposition \ref{p1}, we have
\begin{align*}
E[(X-e_{\tau})_{+}^{2}] &  =2c_{1}(1-\tau)\left(  \gamma^{-1}-1\right)
(1+\eta(\tau))q_{\tau}^{2}\left(  \gamma^{-1}-1\right)  ^{-2\gamma}\left(
1+r(\tau)\right)  ^{2}\left(  1+c_{2}A\left(  \frac{1}{1-\tau}\right)
(1+o(1))\right)  \\
&  =2c_{1}(1-\tau)\left(  \gamma^{-1}-1\right)  q_{\tau}^{2}\left(
1+\eta(\tau)+2r(\tau)+c_{2}A\left(  \frac{1}{1-\tau}\right)  (1+o(1))\right)
.
\end{align*}
Now, we turn to $E(X-e_{\tau})^{2}$. By Proposition \ref{p1}, we have
\begin{align*}
E\left[  (X-e_{\tau})^{2}\right]   &  =q_{\tau}^{2}\left(  \gamma
^{-1}-1\right)  ^{-2\gamma}\left(  1+2r(\tau)\right)  -2q_{\tau}\left(
\gamma^{-1}-1\right)  ^{-\gamma}{E}\left[  X\right]  (1+o(1))\\
&  =q_{\tau}^{2}\left(  \gamma^{-1}-1\right)  ^{-2\gamma}\left(
1+2r(\tau)-2q_{\tau}^{-1}\left(  \gamma^{-1}-1\right)  ^{\gamma}{E}\left[
X\right]  (1+o(1))\right)  .
\end{align*}
Now, plugging the above expansions of $E[(X-e_{\tau})_{+}^{2}]$ and $E\left[
(X-e_{\tau})^{2}\right]  $ into relation (\ref{rw}), we obtain the
second-order expansions
\begin{equation}
dev_{\tau}^{2}(X)=\beta_{\gamma}^{2}q_{\tau}^{2}\left(  1+2w(\tau)\right)
.\label{2ndavar}%
\end{equation}
Now, taking the square root on both sides of (\ref{2ndavar}), the desired result follows.
\end{proof}

\bigskip

\begin{proof}[Proof of Theorem \ref{int}]
Based on the definition of the deviatile and the
second-order asymptotic expansions of the deviatile
established in Theorem \ref{2ndexp}, we have %EDITOR'S NOTE: Please
                                %ensure that the intended meaning has
                                %been maintained in this edit.
\begin{align*}
\frac{\widehat{dev}_{\tau_{n}}}{dev_{\tau_{n}}}  &  =\frac{\widehat{\beta
}_{\gamma}\widehat{q}_{\tau_{n}}}{\beta_{\gamma}q_{\tau_{n}}\left(
1+w(\tau_{n})\right)  }\\
&  =\frac{\frac{(\widehat{\gamma}^{-1}-1)^{-\widehat{\gamma}}}{\sqrt
{1-2\widehat{\gamma}}}\widehat{q}_{\tau_{n}}}{\frac{(\gamma^{-1}-1)^{-\gamma}%
}{\sqrt{1-2\gamma}}q_{\tau}\left(  1+w(\tau_{n})\right)  }\\
&  =\frac{(\widehat{\gamma}^{-1}-1)^{-\widehat{\gamma}}}{(\gamma
^{-1}-1)^{-\gamma}}\times \sqrt{\frac{1-2\gamma}{1-2\widehat{\gamma}}}%
\times \frac{\widehat{q}_{\tau_{n}}}{q_{\tau_{n}}}\times \frac{1}{1+w(\tau_{n}%
)}.
\end{align*}
It then follows that
\begin{align*}
\log \frac{\widehat{dev}_{\tau_{n}}}{dev_{\tau_{n}}}  &  =\log \frac
{(\widehat{\gamma}^{-1}-1)^{-\widehat{\gamma}}}{(\gamma^{-1}-1)^{-\gamma}%
}+\frac{1}{2}\log \frac{1-2\gamma}{1-2\widehat{\gamma}}+\log \frac{\widehat
{q}_{\tau_{n}}}{q_{\tau_{n}}}+\log \frac{1}{1+w(\tau_{n})}\\
=  &  \left[  -\widehat{\gamma}\log(\widehat{\gamma}^{-1}-1)+\gamma \log
(\gamma^{-1}-1)\right]  +\frac{1}{2}\left[  \log \left(  1-2\gamma \right)
-\log(1-2\widehat{\gamma})\right] \\
&  +\log \frac{\widehat{q}_{\tau_{n}}}{q_{\tau_{n}}}-\log \left(  1+w(\tau
_{n})\right)  .\\
&  \triangleq I_{1}+I_{2}+I_{3}-I_{4}.
\end{align*}
We next investigate the asymptotic distribution of the above four terms.

For $I_{1}$, note that
\[
\frac{\mathrm{d}}{\mathrm{d}x}(-x\log(x^{-1}-1))=-\log(x^{-1}-1)+\frac{1}
{1-x}.
\]
By the delta method, we have that as $n\rightarrow \infty$
\[
\sqrt{n\left(  1-\tau_{n}\right)  }I_{1}\overset{d}{\longrightarrow}\left(
-\log(\gamma^{-1}-1)+\frac{1}{1-\gamma}\right)  \Gamma.
\]
For $I_{2}$, note that
\[
\frac{\mathrm{d}}{\mathrm{d}x}\left(  \frac{1}{2}\log \left(  1-2x\right)
\right)  =-\frac{1}{1-2x}.
\]
By the delta method, we have that as $n\rightarrow \infty$
\[
\sqrt{n\left(  1-\tau_{n}\right)  }I_{2}\overset{d}{\longrightarrow}\frac
{1}{1-2\gamma}\Gamma.
\]
For $I_{3}$, by the delta method, we obtain that as $n\rightarrow \infty$
\[
\sqrt{n\left(  1-\tau_{n}\right)  }I_{3}\overset{d}{\longrightarrow}\Theta.
\]
For $I_{4}$, note that by $w(\tau_{n})\rightarrow0$ as $\tau_{n}
\rightarrow \infty$, we have
\[
\log \left(  1+w(\tau_{n})\right)  =w(\tau_{n})(1+o(1))=\xi_{1}q_{\tau_{n}}^{-1}
(1+o(1))+\xi_{2}A\left(  (1-\tau_n)^{-1}\right)  (1+o(1)).
\]
$\text{Then by }\sqrt{n\left(  1-\tau_{n}\right)  }q_{\tau_{n}}^{-1}
\rightarrow \lambda_{1}\in \mathbb{R}$ and $\sqrt{n\left(  1-\tau_{n}\right)
}A\left(  \left(  1-\tau_{n}\right)  ^{-1}\right)  \rightarrow \lambda_{2}
\in \mathbb{R}$,$\text{ we have as }n\rightarrow \infty$
\[
\sqrt{n\left(  1-\tau_{n}\right)  }I_{4}\rightarrow \lambda_{1}\xi_{1}
+\lambda_{2}\xi_{2}.
\]
Finally, combining the above results yields the desired result.
\end{proof}
\bigskip

{
\begin{proof}[Proof of Corollary \ref{cor1}]
It suffices to show the relation (\ref{limit}) holds when $\gamma$ is
estimated by the Hill estimator. To make the notation simple, we set
$k=\left \lfloor n\left(  1-\tau_{n}\right)  \right \rfloor $ in this proof. Let
$Y_{1},Y_{2},...$ be i.i.d. random variables with common distribution function
$1-1/y$, $y\geq1$. Then $U(Y_{i})\overset{d}{=}X_{i}$ for $i=1,2,...$.
Following the proof of Theorem 3.2.5 of \cite{de2007extreme},
\begin{align*}
&  \sqrt{k}\left(  \widehat{\gamma}_{H}-\gamma \right)  \overset{d}{=}%
\gamma \sqrt{k}\left(  \frac{1}{k}\sum_{i=0}^{k-1}\log \left(  \frac{Y_{n-i,n}%
}{Y_{n-k,n}}\right)  -1\right)  +\sqrt{k}A\left(  Y_{n-k,n}\right)  \frac
{1}{k}\sum_{i=0}^{k-1}\frac{\left(  \frac{Y_{n-i,n}}{Y_{n-k,n}}\right)
^{\rho}-1}{\rho}\\
&  +o_{p}(1)\sqrt{k}A\left(  Y_{n-k,n}\right)  \frac{1}{k}\sum_{i=0}%
^{k-1}\left(  \frac{Y_{n-i,n}}{Y_{n-k,n}}\right)  ^{\rho+\varepsilon}.
\end{align*}
The second term above converges in probability to $\lambda_{2}/(1-\rho)$ and
the third term vanishes. Thus, to show the limiting relation of (\ref{limit})
is suffices to show
\begin{equation}
\sqrt{k}\left(  \gamma \left(  \frac{1}{k}\sum_{i=0}^{k-1}\log \left(
\frac{Y_{n-i,n}}{Y_{n-k,n}}\right)  -1\right)  ,\frac{U\left(  Y_{n-k,n}%
\right)  }{U\left(  n/k\right)  }-1\right)  \overset{d}{\longrightarrow}(\widetilde{\Gamma
},\Theta).\label{joint}%
\end{equation}
By Lemma 3.2.3 in \cite{de2007extreme}, the limiting distribution
$\widetilde{\Gamma}$ satisfies that
\[
\widetilde{\Gamma}\overset{d}{=}\mathcal{N}\left(  0,\gamma^{2}\right)  ,
\]
and\ by Theorem 2.4.1 in \cite{de2007extreme},
\[
\Theta \overset{d}{=}\mathcal{N}\left(  0,\gamma^{2}\right)  .
\]
By the reasoning of Lemma 3.2.3 in \cite{de2007extreme}, $\left \{
\frac{Y_{n-i,n}}{Y_{n-k,n}}\right \}  _{i=1}^{k-1}$ is an i.i.d. sequence with
the same common distribution as $Y_{1}$ and it is independent of $Y_{n-k,n}$.
Thus $\widetilde{\Gamma}$ and $\Theta$ are independent, and $\Gamma
=\widetilde{\Gamma}+\lambda_{2}/(1-\rho)$ is independent with $\Theta$ as
well. By using the Cramer-Wold's theorem,
the joint convergence in (\ref{joint}) holds and hence the relation
(\ref{limit}) holds. 
\end{proof}}
\bigskip

\begin{proof}[Proof of Theorem \ref{extreme}]
Based on the ratio of $\widehat{q}_{p_{n}}$ and $\widehat{q}_{\tau_{n}}$, we
denote that
\[
\widehat{q}_{p_{n}}^{\ast}=\left(  \frac{1-p_{n}}{1-\tau_{n}}\right)
^{-\widehat{\gamma}}\widehat{q}_{\tau_{n}}.
\]
By the definition of $\widehat{dev}_{p_{n}}^{\ast}$, we have
\begin{align*}
\log \left(  \frac{\widehat{dev}_{p_{n}}^{\ast}}{dev_{p_{n}}}\right)   &
=\log \left(  \frac{\left(  \frac{1-p_{n}}{1-{\tau_{n}}}\right)  ^{-\widehat
{\gamma}}\widehat{dev}_{\tau_{n}}}{{dev}_{p_{n}}}\right)  \\
&  =\log \left(  \frac{\widehat{q}_{p_{n}}^{\ast}}{q_{p_{n}}}\right)
+\log \left(  \frac{\widehat{dev}_{\tau_{n}}}{dev_{\tau_{n}}}\right)
-\log \left(  \frac{\widehat{q}_{\tau_{n}}}{q_{\tau_{n}}}\right)  +\log \left(
\frac{dev_{\tau_{n}}}{q_{\tau_{n}}}\right)  -\log \left(  \frac{dev_{p_{n}}%
}{q_{p_{n}}}\right)  .
\end{align*}
Next, we analyze each term above. By Theorem 4.3.8 of \cite{de2007extreme}, since
$\log \left[  \left(  1-\tau_{n}\right)  /\left(  1-p_{n}\right)  \right]
\rightarrow \infty$, we have
\begin{equation}
\frac{\sqrt{n\left(  1-\tau_{n}\right)  }}{\log \left[  \left(  1-\tau
_{n}\right)  /\left(  1-p_{n}\right)  \right]  }\log \left(  \frac{\widehat
{q}_{p_{n}}^{\ast}}{q_{p_{n}}}\right)  \overset{d}{\longrightarrow}%
\Gamma.\label{t1}%
\end{equation}
By Theorem \ref{int}, we have that as $\log \left[  \left(  1-q_{n}\right)  /\left(
1-p_{n}\right)  \right]  \rightarrow \infty$,
\begin{equation}
\frac{\sqrt{n\left(  1-\tau_{n}\right)  }}{\log \left[  \left(  1-\tau
_{n}\right)  /\left(  1-p_{n}\right)  \right]  }\log \left(  \frac
{\widehat{dev}_{\tau_{n}}}{dev_{\tau_{n}}}\right)  =o_{p}(1).\label{t2}%
\end{equation}
By Theorems 2.4.1 of \cite{de2007extreme}, we have that as $\log \left[  \left(
1-q_{n}\right)  /\left(  1-p_{n}\right)  \right]  \rightarrow \infty$
\begin{equation}
\frac{\sqrt{n\left(  1-\tau_{n}\right)  }}{\log \left[  \left(  1-\tau
_{n}\right)  /\left(  1-p_{n}\right)  \right]  }\log \left(  \frac{\widehat
{q}_{\tau_{n}}}{q_{\tau_{n}}}\right)  =o_{p}(1).\label{t3}%
\end{equation}
Lastly, note that by Theorem \ref{2ndexp}, we have
\begin{align*}
&  \log \left(  \frac{dev_{\tau_{n}}}{q_{\tau_{n}}}\right)  -\log \left(
\frac{dev_{p_{n}}}{q_{p_{n}}}\right)  \\
&  =\log \left(  \frac{\beta_{\gamma}(1+w(\tau_{n}))/q_{\tau_{n}}}%
{\beta_{\gamma}(1+w(p_{n}))/q_{p_{n}}}\right)  \\
&  =\log \left(  1+w(\tau_{n})(1+o(1))-w(p_{n})(1+o(1))\right)  \\
&  =w(\tau_{n})(1+o(1))-w(p_{n})(1+o(1))\\
&  =\xi_{1}q_{\tau_{n}}^{-1}\left(  1-\left(  \frac{1-p_{n}}{1-\tau_{n}%
}\right)  ^{\gamma}\right)  (1+o(1))+\xi_{2}A\left(  (1-\tau_{n})^{-1}\right)
\left(  1-\left(  \frac{1-\tau_{n}}{1-p_{n}}\right)  ^{\rho}\right)
(1+o(1))).
\end{align*}
Then, by $\sqrt{n\left(  1-\tau_{n}\right)  }q_{\tau_{n}}^{-1}\rightarrow
\lambda_{1}\in \mathbb{R}$ and $\sqrt{n\left(  1-\tau_{n}\right)  }%
A((1-\tau_{n})^{-1})\rightarrow \lambda_{2}\in \mathbb{R}$, and $\log \left[
\left(  1-q_{n}\right)  /\left(  1-p_{n}\right)  \right]  \rightarrow \infty$,
we have
\begin{equation}
\frac{\sqrt{n\left(  1-\tau_{n}\right)  }}{\log \left[  \left(  1-\tau
_{n}\right)  /\left(  1-p_{n}\right)  \right]  }\left(  \log \left(
\frac{dev_{\tau_{n}}}{q_{\tau_{n}}}\right)  -\log \left(  \frac{dev_{p_{n}}%
}{q_{p_{n}}}\right)  \right)  =o_{p}(1).\label{t4}%
\end{equation}
Combing (\ref{t1}) to (\ref{t4}), we obtain the desired result.
\end{proof}
\bigskip

{
\begin{proof}[Proof of Lemma \ref{dep int}]
By Proposition 1 of \cite{chavez2018extreme}, as $n\rightarrow \infty$,
\[
\sqrt{k}\left(  \log \frac{Q_{n}(t)}{U\left(  \frac{n}{k}\right)  }+\gamma \log
t\right)  =\gamma t^{-1}e(t)+\sqrt{k}\widetilde{A}\left(  \frac{n}{k}\right)
\frac{t^{-\rho}-1}{\rho}+o_{p}(1)t^{-1/2-\varepsilon},
\]
where $e(t)$ is a centered Gaussian process with covariance function
$r(\cdot,\cdot)$, $\widetilde{A}(x)\sim A(x)$ as $x\rightarrow \infty$ and
$o_{p}(1)$ is uniform for $t\in \lbrack0,1]$. Thus,
\begin{align*}
\sqrt{k}\left(  \widehat{\gamma}_{H}-\gamma \right)   &  =\sqrt{k}\left(
\int_{0}^{1}\log \frac{Q_{n}(t)}{Q_{n}(1)}dt+\int_{0}^{1}\gamma \log tdt\right)
\\
&  =\gamma \int_{0}^{1}t^{-1}e(t)-e(1)dt+\sqrt{k}\widetilde{A}\left(  \frac
{n}{k}\right)  \int_{0}^{1}\frac{t^{-\rho}-1}{\rho}dt+o_{p}(1)\int_{0}%
^{1}t^{-1/2-\varepsilon}dt\\
&  =\gamma \int_{0}^{1}t^{-1}e(t)-e(1)dt+\sqrt{k}\widetilde{A}\left(  \frac
{n}{k}\right)  \frac{1}{1-\rho}+o_{p}(1).
\end{align*}
The asymptotic variance of $\sqrt{k}\left(  \widehat{\gamma}_{H}%
-\gamma \right)$ is calculated as
\begin{align*}
\sigma_{H,\gamma}^{2} &  =var\left(  \gamma \int_{0}^{1}t^{-1}%
e(t)-e(1)dt\right)  \\
&  =\gamma^{2}\int_{0}^{1}\int_{0}^{1}E\left[  \left(  t^{-1}e(t)-e(1)\right)
\left(  s^{-1}e(s)-e(1)\right)  \right]  dsdt\\
&  =\gamma^{2}\left(  \int_{[0,1]^{2}}\left(  st\right)  ^{-1}r(t,s)dsdt-2\int
_{0}^{1}t^{-1}r(t,1)dt+r(1,1)\right)  .
\end{align*}
Since $\sqrt{k}\widetilde{A}\left(  \frac{n}{k}\right)  \rightarrow \lambda$,
we have $\sqrt{k}\left(  \widehat{\gamma}_{H}-\gamma \right)  \rightarrow
N\left(  \frac{\lambda}{1-\rho},\sigma_{H,\gamma}^{2}\right)  $ as
$n\rightarrow \infty$. Similarly, we have \
\begin{align*}
\sqrt{k}\left(  \frac{\widehat{q}_{\tau_{n}}}{q_{\tau_{n}}}-1\right)   
=\sqrt{k}\left(  \frac{Q_{n}(1)}{U\left(  \frac{n}{k}\right)  }-1\right) 
 =\gamma e(1)+o_{p}(1).
\end{align*}
And as $n\rightarrow \infty$, $\sqrt{k}\left(  \frac{\widehat{q}_{\tau_{n}}%
}{q_{\tau_{n}}}-1\right)  \rightarrow N(0,\gamma^{2}r(1,1))$. The convergence of $\sqrt{k}\left(  \widehat{\gamma}_{H}-\gamma,\frac
{\widehat{q}_{\tau_{n}}}{q_{\tau_{n}}}-1\right)  $ holds jointly by having the
same Gaussian process $e$. This verifies the joint convergence of relation \eqref{di}. The covariance of the limiting
distribution can be calculated as
\begin{align*}
\varsigma &  =cov\left(  \gamma \int_{0}^{1}t^{-1}e(t)-e(1)dt,\gamma
e(1)\right)  \\
&  =E\left[  \left(  \gamma \int_{0}^{1}t^{-1}e(t)-e(1)dt\right)  \gamma
e(1)\right]  \\
&  =\gamma^{2}\int_{0}^{1}E\left[  t^{-1}e(t)e(1)-e(1)^{2}\right]  dt\\
&  =\gamma^{2}\left(  \int_{0}^{1}t^{-1}r(t,1)dt-r(1,1)\right)  .
\end{align*}
This ends the proof.
\end{proof}
}

\bigskip
\section*{Acknowledgments}

The authors gratefully acknowledge two anonymous referees for their helpful comments which resulted in a substantially improved version of this article. T. Mao gratefully acknowledges the financial support from the Natural Science Foundation of Anhui Province (No. 2208085MA07) and the National Natural Science Foundation of China (Grants 12371476, 71671176, and 71921001). F.~Yang gratefully acknowledges financial support from the Natural Sciences and Engineering Research Council of Canada (Grant Number: 04242).

\begin{comment}
{\color{blue}
\begin{proof}[Proof of Lemma \ref{dep int}]
By Corollary 3.3 of \cite{drees2000weighted}, we have
\begin{equation}
\sqrt{k}\left(  T_{H}(Q_{n})-\gamma \right)  \rightarrow N\left(  \frac
{\lambda}{1-\rho},\sigma_{T,\gamma}^{2}\right)  ,\label{hill}%
\end{equation}
where $\sigma_{T,\gamma}^{2}=\int_{[0,1]^{2}}(st)^{\gamma-1}r(s,t)dsdt$. By
Theorem 3.1 and relation (3.7) of \cite{drees2000weighted}, we have
\begin{equation}
\sqrt{k}\left(  \frac{Q_{n}(1)}{F^{\leftarrow}(1-\frac{k}{n})}-1\right)
\rightarrow \gamma e_{\gamma}(1),\label{qtl}%
\end{equation}
where $e_{\gamma}(t)=t^{-(\gamma+1)}e(t)$ with $e(t)$ a centered continuous
Gaussian process with the covariance function $r(s,t)$. Thus, $\gamma
e_{\gamma}(1)=\gamma N(0,r(1,1))$. By Cramer-Wold's theorem, if we can show
that for any $a,b\in \mathbb{R}$,
\[
\sqrt{k}\left(  a(T_{H}(Q_{n})-\gamma)+b\left(  \frac{Q_{n}(1)}{F^{\leftarrow
}(1-\frac{k}{n})}-1\right)  \right)  \rightarrow aN\left(  \frac{\lambda
}{1-\rho},\sigma_{T,\gamma}^{2}\right)  +bN(0,\gamma^{2}r(1,1)),
\]
then (\ref{di}) holds true. Note that $T_{\gamma}^{\prime H}(e_{\gamma
})=N\left(  0,\sigma_{T,\gamma}^{2}\right)  $, where $T_{\gamma}^{\prime H}$
is the Hadamard derivative of $T_{H}$. Now combining Theorem 3.1 and relation
(3.7) of \cite{drees2000weighted} and the functional delta method (see, e.g. \citealp{van2000asymptotic}), the desired
result follows.
\end{proof}
}
    
\end{comment}
	\bibliographystyle{apalike}
	\bibliography{reference}
\end{document}